\documentclass[12pt,draftcls,onecolumn]{IEEEtran}

\usepackage{color}
\usepackage{graphicx}
\usepackage{amsfonts,amssymb}
\usepackage{bm}
\usepackage{dsfont}
\usepackage{amsmath}
\usepackage{amsthm}
\usepackage{mathrsfs}
\usepackage{cases}
\usepackage{arydshln}
\usepackage{empheq}
 \usepackage{subcaption}
\usepackage{mathrsfs}
\usepackage[colorlinks,
linkcolor=cyan,
anchorcolor=blue,
citecolor=red
]{hyperref}
\usepackage{cite}
\usepackage{tcolorbox}
\usepackage{algorithmicx}
\usepackage{algorithm}
\usepackage{algpseudocode}
\algnewcommand\algorithmicinput{\textbf{Input:}}
\algnewcommand\Input{\item[\algorithmicinput]}
\algnewcommand\algorithmicoutput{\textbf{Output:}}
\algnewcommand\Output{\item[\algorithmicoutput]}

\newcommand{\blkdiag}{\mathtt{blkdiag}}

\newtheorem{thm}{Theorem}
\newtheorem{exm}{Example}
\newtheorem{lem}{Lemma}
\newtheorem{rem}{Remark}
\newtheorem{pro}{Proposition}

\newtheorem{defn}{Definition}
\newtheorem{coro}{Corollary}
\newtheorem{prob}{Problem}

\begin{document}

\title{\LARGE \bf
	Gramian-Based Model Reduction of Directed Networks
}

\author{Xiaodong~Cheng and
	 and Jacquelien M.A. Scherpen
	\thanks{Xiaodong Cheng is with Control Systems Group, Department of Electrical Engineering,
		Eindhoven University of Technology,
		5600 MB Eindhoven,
		The Netherlands
		{\tt\small x.cheng@tue.nl}}
	\thanks{Jacquelien M. A. Scherpen is with Jan C. Willems Center for Systems and Control, Engineering and Technology Institute Groningen, Faculty of Science and Engineering, University of Groningen, Nijenborgh 4, 9747 AG Groningen, the Netherlands.
		{\tt\small j.m.a.scherpen@rug.nl}}%
}

\maketitle

\begin{abstract}
	
This paper investigates a model reduction problem for linear directed network systems, in which the interconnections among the vertices are described by general weakly connected digraphs. First, the definitions of pseudo controllability and observability Gramians are proposed for semistable systems, and their solutions are characterized by Lyapunov-like equations. Then, we introduce a concept of vertex clusterability to guarantee the boundedness of the approximation error and use the newly proposed Gramians to facilitate the evaluation of the dissimilarity of each pair of vertices. An clustering algorithm is thereto provided to generate an appropriate graph clustering, whose characteristic matrix is employed as the projections in the Petrov-Galerkin reduction framework. The obtained reduced-order system preserves the weakly connected directed network structure, and the approximation error is computed by the pseudo Gramians. Finally, the efficiency of the proposed approach is illustrated by numerical examples.
 
\end{abstract}

\section{Introduction} 

A network system captures the behaviors of a collection of dynamical subsystems. In recent decades, the study of such systems gradually becomes a popular topic of interdisciplinary research, which appears in e.g., social and ecological interactions, chemical reactions and physical networks, see e.g. \cite{Newman2010NetworksIntroduction,Higham2008ChemicalReactions,boccaletti2006complex,Schaft2017PhysicalNetwork,XiaodongACOM2018Power} for an overview. An important property of network systems is consensus, which occurs when certain agreements are reached via exchanging the information among the vertices \cite{ren2005survey}. Formation control of mobile vehicles, current sharing and voltage regulation in DC microgrids, coordination of distributed sensors, and balancing
in chemical kinetics can be viewed as different applications of network consensus\cite{jadbabaie2003coordination,Scutari2008SensorNetworks,cucuzzella2018robust,Rao2015balancing}.
The model reduction of consensus network systems is motivated by the challenges of their large scale and high complexity that cause limitations for both theoretical analysis and experimental investigations. It is worth addressing a structure preserving model reduction problem, which aims to derive a lower-dimensional model that can approximate the behavior of the original network with an acceptable accuracy and without being too expensive to evaluate. Furthermore, in the reduction process, it is desirable to preserve a network structure, since such a structure determines the consensus property of a network and is essential for the further applications of reduced-order models to e.g., distributed controller design or sensor allocation.

A variety of techniques are available in the literature for reducing the dimension of a linear state-space system. Classic approaches include balanced truncation, Hankel norm approximation, and Krylov subspace methods, see e.g., \cite{moore1981principal,antoulas2005approximation,astolfi2010model}. They provide systematic procedures to generate reduced-order models that approximate the input-output characteristics of the original large-scale systems. 
Nevertheless, a direct application of these conventional methods may not maintain a network architecture in the reduced-order model, as they do not impose any structure for the reduced state space. Consequently, the states of network vertices are mixed and thus lose a network interpretation. In \cite{XiaodongBT2017}, a structure preserving method is developed using the generalized balanced truncation for undirected networks. Even though it yields a reduced-order model that can be realized as a network system by a proper coordinate transformation, the topology relation between  the original network and the reduced one is no longer clear.


Recently, graph clustering (or graph partition) has shown a great potential in the structure preserving model reduction of network systems. By assimilating the vertices in each cluster into a single vertex, the essential information of the original topology can be retained. It has to be emphasized that the idea of grouping vertices is relevant to the problem of community or cluster detection in static networks, see e.g., \cite{Aggarwal2013Clustering,Schaeffer2007SurveyClustering}.
For dynamical networks that exhibit consensus properties, the clustering process has to take into account the evolution of vertex states driven by external excitation and disturbance signals.

For dynamical systems on undirected networks, the methods developed by \cite{Monshizadeh2014,Petar2015CDC,Ishizaki2014} formulate the model reduction problem in the Petrov-Galerkin framework, and the projections are generated from selected graph clusterings. 
However, many applications are considering directed networks, e.g., chemical reaction networks \cite{Higham2008ChemicalReactions} or metabolic processes \cite{Banasiak2015EvolutionaryEquations}, where the mass/energy exchanges among different species are usually directional.

A pioneering approach dealing with semistable directed networks is proposed in \cite{ishizaki2015clustereddirected}. 
The graph clustering is formed based on a notion of cluster reducibility, characterized by the uncontrollability of local states. Merging the vertices in the reducible clusters then yields a reduced-order model that preserves the structural information of a directed network.
However,  the projection of this method heavily relies on the assumption that the network matrix has only one simple zero eigenvalue, whose corresponding eigenvectors (i.e., the \textit{Frobenius eigenvectors}) has all strictly positive entries.  An alternative approach in \cite{XiaodongCDC2017Digraph} focuses on the behavior of individual vertices described by transfer functions and pairwise dissimilarities evaluated by function norms. The vertices behaving similarly are sequentially assimilated to a single vertex. This approach is preferable for a consensus network and applicable to a strongly connected topology. In broader applications of dynamical networks, for instance, biochemical systems, sensor coordination, gene regulation, weakly connected spatial structures commonly appear in the networks, see e.g., \cite{Gunawardena2009nature,Scutari2008SensorNetworks,Ahsendorf2014GeneRegulation}.

This motivates us to consider weakly connected directed networks. As undirected networks and strongly connected networks are only subcategories of weakly connected ones, the systems studied in this paper describe more general scenarios, and the proposed method can be also applied to the former two cases. It is worth noting that a model reduction problem of weakly connected directed networks has been absent from the literature so far. The major difficulty for such networks is an appropriate clustering selection scheme. The approximation accuracy heavily relies on the resulting graph clustering, whereas finding an optimal clustered network is roughly an NP-hard problem even for static networks \cite{Aggarwal2013Clustering,SurveyClustering}.
More importantly, in \cite{XiaodongCDC2017Digraph,ishizaki2015clustereddirected}, projections are generated using the positive Frobenius eigenvectors of the system matrix. However, such vectors may not exist in the weakly connected case. Furthermore, a weakly connected network may not reach a global consensus as strongly connected ones do. Instead, a local consensus is achievable among the vertices that are able to influence each other. Consequently, the clustering for a weakly connected graph has to be prudently selected to avoid an unbounded approximation error. 
 
To tackle the above difficulties, this paper introduces a definition of vertex clusterability for weakly connected networks and shows that the boundedness of the approximation error is guaranteed if and only if clusterable vertices are aggregated. Thereby, the concept of dissimilarity is defined only for clusterable vertices. In contrast to \cite{XiaodongECC2016,XiaodongCDC2017Digraph}, the input and output dissimilarities are considered based on the responses of the vertex states to the external inputs and the measurement of the state discrepancy from the output channels, respectively. Thus, the pairwise dissimilarities are evaluated by combining the input and output efforts. 
Then, according to the vertex clusterability and dissimilarity, a graph cut algorithm is designed to partition the underlying network into a desired number of clusters. Then, a clustering-based projection is employed to reduce the dimension of the original network system, where the projection matrix is generated from the left kernel space of the system matrix. The proposed method yields a reduced-order model that preserves not only the structure and connectedness of directed network but several fundamental properties, including consensus, semistability, and asymptotic behaviors of the vertex.
 
Another contribution of this paper is to summarize the notion of controllability Gramians in \cite{XiaodongCDC2016Gramian,ishizaki2015clustereddirected} and extend the results to propose a pair of pseudo controllability and observability Gramians for general semistable systems. The new Gramians can be viewed as the generalization of standard Gramians for asymptotically stable systems. Moreover, the pseudo Gramians are   characterized by a set of Lyapunov equations, and their ranks are strongly related to the controllability and the observability of a semistable system. Using the pseudo Gramians, the $\mathcal{H}_2$-norm of a semistable system can be easily evaluated. Therefore, this paper employs them to facilitate the computation of input and output dissimilarities and thus provides a crucial step in the clustering-based model reduction.
  

The rest of this paper is organized as follows: In Section \ref{sec:Gramian}, we introduce the definition of pseudo Gramians for semistable systems, and some important properties of the new Gramians are the discussed. Section \ref{sec:networksystem} presents the model of directed networks, and the Petrov-Galerkin reduction framework is established based on graph clustering. Then, in Section \ref{sec:Reduction}
we define the vertex clusterability and dissimilarity, and propose a scheme for model reduction of directed networks. The proposed method is illustrated through an example in Section \ref{sec:example}, and finally, concluding remarks are made in Section \ref{sec:conclusion}.

\textit{Notation:}
Denote $\mathbb{R}$ as the set of real numbers and $\mathbb{R}_+$ as set of real nonnegative numbers. $\mathbb{R}^n$ is a vector space of $n$ dimension. Let $\mathbb{W}$ be a subspace of $\mathbb{R}^n$, then $\mathbb{W}^\perp$ denotes the orthogonal complement of $\mathbb{W}$ in $\mathbb{R}^n$.
The cardinality of a set $\mathbb{V}$ is denoted by $\lvert \mathbb{V} \rvert$, and $\mathsf{dim} (\mathbb{W})$ represents the dimension of space $\mathbb{W}$
The identity matrix of size $n$ is given as $I_n$, and $\mathds{1}_n$ denotes a $n$-entries vector of all ones. The subscript $n$ is omitted when no confusion arises. $\mathbf{e}_i$ is the $i$-th column vector of $I_n$, and $\mathbf{e}_{ij} = \mathbf{e}_i - \mathbf{e}_j$. The trace, rank, image and nullspace of $A$ are denoted by $\mathsf{tr}(A)$, $\mathsf{rank}(A)$, $\mathsf{im}(A)$, and $\mathsf{ker}(A)$, respectively.  
%
%
%
%
\section{Gramians of Semistable System} \label{sec:Gramian}

We make the result of this section self-contained and independent of the model reduction of directed network systems. 
This section extends the definition of \textit{controllability and observability Gramians} from asymptotically stable systems to semistable ones. In our preliminary results in \cite{XiaodongCDC2016Gramian,XiaodongTAC20172OROM}, new Gramians are introduced for first-order and second-order network systems. Here, we present a generalization of the results to general semistable systems.

Consider the state-space model of a linear time-invariant system 
\begin{equation} \label{syss}
\left \{
\begin{array}{l}
\dot{x}(t) = A x(t) + Bu(t), \\
y(t) = C x(t),
\end{array}
\right.
\end{equation}
with states $x \in \mathbb{R}^n$, inputs $u \in \mathbb{R}^p$ and outputs $y \in \mathbb{R}^q$. 
\begin{defn} \label{defn:SemiSys}\cite{Semistable}
	The system \eqref{syss} is \textbf{semistable} if $\lim\limits_{t \rightarrow \infty} x(t)$ exists for all initial states $x(0)$ and $u(t) = 0$. 
\end{defn}
A necessary and sufficient condition for the semistability of $\bm{\Sigma}$ is provided in \cite{Semistable}.
\begin{lem} \label{lem:semisimplestable}
	The system \eqref{syss} is semistable if and only if the zero eigenvalues of $A$ in \eqref{syss} are \textbf{semisimple}, i.e., the geometric multiplicity of the zero eigenvalue coincides with the algebraic multiplicity, and all the other eigenvalues have negative real parts. 
\end{lem}


Generally, semistable systems are not in the $\mathcal{H}_2$ space, meaning that the standard controllability and observability Gramians in \cite{antoulas2005approximation} are not well-defined for the semistable case. 
Thus, we propose new definitions of Gramians for semistable systems.
\begin{defn} \label{defn:ConGram}
	Consider the semistable system  \eqref{syss}. The \textbf{pseudo controllability and observability Gramians} are defined as 
	\begin{subequations} \label{defn:PseudoGramians}
		\begin{equation} \label{defn:P}
		\mathcal{P} = \int_{0}^{\infty} 
		(e^{A\tau}-\mathcal{J}) BB^\top (e^{A^\top\tau}-\mathcal{J}^\top) \mathrm{d}\tau \in \mathbb{R}^{n \times n},
		\end{equation}
		\begin{equation} \label{defn:Q}
		\mathcal{Q} = \int_{0}^{\infty} 
		(e^{A^\top\tau}-\mathcal{J}^\top) C^\top C (e^{A \tau}-\mathcal{J}) \mathrm{d}\tau \in \mathbb{R}^{n \times n},
		\end{equation}
	\end{subequations}
	where $\mathcal{J} : = \lim\limits_{\tau \rightarrow \infty} e^{A\tau}$ is a constant matrix.
\end{defn} 

Note that the pseudo Gramians in \eqref{defn:P} and \eqref{defn:Q} are well-defined, since the integrands in both integral converge to zero when $\tau \rightarrow \infty$. Furthermore, using the matrix $\mathcal{J}$, the Lyapunov characteristics of $\mathcal{P}$ and $\mathcal{Q}$ in Definition 
\ref{defn:ConGram} are provided.
\begin{thm} \label{thm:GenGram}
	Consider the semistable system  \eqref{syss}. The pseudo controllability and observability Gramians of $\bm{\Sigma_s}$, $\mathcal{P}$ and $\mathcal{Q}$ defined in \eqref{defn:PseudoGramians}, are the unique symmetric solutions of
	the following sets of linear matrix equations
	\begin{subequations} \label{eq:thm2eqs}
		\begin{empheq}[left=\empheqlbrace]{align}
		0 & = A \mathcal{P} + \mathcal{P} A^\top + (I-\mathcal{J})BB^\top(I-\mathcal{J}^\top),  \label{eq:LyapLike} \\
		0 & = \mathcal{J} \mathcal{P} \mathcal{J}^\top.
		\label{eq:VPV}
		\end{empheq}
	\end{subequations}
	\begin{subequations} \label{eq:thm2eqsQ}
		\begin{empheq}[left=\empheqlbrace]{align}
		0 &= A^\top \mathcal{Q} + \mathcal{Q} A  + (I-\mathcal{J}^\top)C^\top C(I-\mathcal{J}),  \label{eq:LyapLikeQ} \\
		0 &= \mathcal{J}^\top \mathcal{Q} \mathcal{J}. \label{eq:UQU}
		\end{empheq}	
	\end{subequations}
	
\end{thm}

\begin{proof}
	Assume that $A$ in \eqref{syss} has zero eigenvalues with the geometric (or algebraic) multiplicity $m$, which means that the eigenspace of the zero eigenvalues has dimension $m$. Therefore, there exists a similarity transformation 
	\begin{equation} \label{eq:decompostion}
	A = \mathcal{U} D \mathcal{U}^{-1}
	=\begin{bmatrix}
	U & \bar{U}
	\end{bmatrix}  \begin{bmatrix}
	{0}_{m \times m} & \\
	&  \bar{A}\\ 
	\end{bmatrix}
	\begin{bmatrix}
	V^\top\\
	\bar{V}^\top\\ 
	\end{bmatrix},
	\end{equation} 
	such that $\bar{A} \in \mathbb{R}^{(n-m) \times (n-m)} $ is Hurwitz, and the matrices $U \in \mathbb{R}^{n \times m}$ and $V \in \mathbb{R}^{n \times m}$ fulfill
	\begin{equation} \label{eq:UVcond}
	\mathcal{R}(U) = \mathcal{N}(A), \ \mathcal{R}(V) = \mathcal{N}(A^\top),
	\ \text{and} \
	V^\top U = I_m.
	\end{equation}
	Note that the product $U V^\top$ is invariant to the choices for $U$ and $V$, and it coincides with the matrix $\mathcal{J}$ in \eqref{defn:PseudoGramians}, i.e., 
	\begin{equation} \label{defn:convergence}
	\mathcal{J}= \lim\limits_{\tau \rightarrow \infty} e^{A\tau}  = U V^\top.
	\end{equation}
	Therefore, the following equations hold:
	\begin{equation} \label{eq:J2=J}
	\mathcal{J}^2 = \mathcal{J}, \ A \mathcal{J}  = 0, \ \text{and} \
	\mathcal{J} A   = 0.
	\end{equation}
	Furthermore, for any $\tau \in \mathbb{R}$,
	\begin{equation} \label{eq:JeAt}
	\mathcal{J} e^{A\tau} = \mathcal{J} 
	\left(I+\sum_{k=1}^{\infty} 
	\dfrac{A^k\tau^k}{k!}\right) =  \mathcal{J}, \ \text{and} \ e^{A\tau} \mathcal{J} = \mathcal{J}.
	\end{equation}
	Notice that 
	\begin{align} \label{eq:intsplit}
	& \dfrac{d}{d \tau} \left[ (e^{A\tau}-\mathcal{J}) BB^\top (e^{A^\top\tau}-\mathcal{J}^\top) \right] \nonumber\\
	=  &
	A e^{A\tau} BB^\top (e^{A^\top\tau}-\mathcal{J}^\top) + 
	(e^{A\tau}-\mathcal{J}) BB^\top e^{A^\top\tau} A^\top.
	\end{align} 
	Taking the integral of each term leads to  
	\begin{align} \label{eq:part1}
	\int_{0}^{\infty} A e^{A\tau} BB^\top (e^{A^\top\tau}-\mathcal{J}) \mathrm{d}\tau 
	= & \int_{0}^{\infty} A (e^{A\tau} -\mathcal{J} + \mathcal{J}) BB^\top  (e^{A^\top\tau}-\mathcal{J}) \mathrm{d}\tau
	\nonumber \\
	= & A \mathcal{P} 
	+ A\mathcal{J} BB^\top \int_{0}^{\infty} (e^{A^\top\tau}-\mathcal{J}) \mathrm{d}\tau
	= 
	A \mathcal{P},
	\end{align}
	and similarly,
	\begin{equation}
	\begin{split} \label{eq:part2}
	\int_{0}^{\infty} (e^{A\tau}-\mathcal{J}) BB^\top e^{A^\top\tau} A^\top \mathrm{d}\tau
	= \mathcal{P} A^\top,
	\end{split}
	\end{equation}
	Consequently, we obtain
	\begin{align}\label{eq:part3}
	A \mathcal{P} + \mathcal{P} A  
	=&\int_{0}^{\infty} \dfrac{d}{d \tau} \left[ (e^{A\tau}-\mathcal{J}) BB^\top (e^{A^\top\tau}-\mathcal{J}^\top) \right] \mathrm{d}\tau \nonumber \\ 
	=  &\left. (e^{A\tau}-\mathcal{J}) BB^\top (e^{A^\top\tau}-\mathcal{J}^\top) \right|_{0}^{\infty} 
	\nonumber\\
	= & (I-\mathcal{J})BB^\top(I-\mathcal{J}^\top).  
	\end{align}
	%
	The second equation in \eqref{eq:VPV} can be seen from the fact that
	$
	\mathcal{J}  \left(e^{A\tau}-\mathcal{J}\right)  = \mathcal{J} - \mathcal{J} = 0.
	$
	
	%
	Next, we prove the uniqueness of the solution of \eqref{eq:thm2eqs} by contradiction. Assume that two symmetric matrices $\mathcal{P}_1$ and $\mathcal{P}_2$ satisfy \eqref{eq:thm2eqs} and $\mathcal{P}_1 \ne \mathcal{P}_2$.
	From \eqref{eq:LyapLike}, we have
	\begin{equation}
	{A} (\mathcal{P}_1-\mathcal{P}_2) + (\mathcal{P}_1-\mathcal{P}_2) {A}^\top =0, 
	\end{equation}
	which leads to
	\begin{align}
	e^{{A}\tau}\left[{A} (\mathcal{P}_1-\mathcal{P}_2) + (\mathcal{P}_1-\mathcal{P}_2) {A}^\top \right]e^{{A}^\top\tau} 
	=\dfrac{d}{\mathrm{d}\tau}\left[e^{{A}\tau} (\mathcal{P}_1-\mathcal{P}_2) e^{{A}^\top t}\right]= 0.
	\end{align}
	Therefore, 
	$
	\int_{0}^{\infty}\dfrac{d}{\mathrm{d}\tau}\left[e^{{A}\tau} (\mathcal{P}_1-\mathcal{P}_2) e^{{A}^\top\tau}\right] \mathrm{d}\tau= 0,
	$
	which implies that
	\begin{equation} \label{eq:P1P2eq}
	\mathcal{P}_1-\mathcal{P}_2 = \mathcal{J} (\mathcal{P}_1-\mathcal{P}_2) \mathcal{J}^\top.
	\end{equation}
	As both $\mathcal{P}_1$ and $\mathcal{P}_2$ satisfy \eqref{eq:VPV}, the equation \eqref{eq:P1P2eq} becomes zero, which, however,
	contradicts the assumption that $\mathcal{P}_1 \ne \mathcal{P}_2$. Therefore, the common solution of \eqref{eq:LyapLike} and \eqref{eq:VPV} is unique. The proof of the pseudo observability Gramian in \eqref{eq:thm2eqsQ} is similar to the controllability Gramian part, and thus the details are omitted here.
\end{proof}

\begin{rem}
	It is implied by \eqref{eq:VPV} and \eqref{eq:UQU} that the pseudo Gramians $\mathcal{P}$ and $\mathcal{Q}$ are positive semidefinite. Particularly, when $A$ is Hurwitz, i.e., $\bm{\Sigma}$ is asymptotically stable, it follows that $\mathcal{J} = 0$, implying that $\mathcal{P}$ and $\mathcal{Q}$ in \eqref{defn:PseudoGramians} become the standard Gramians. Thus, the pseudo Gramians are generalizations of the standard ones.
\end{rem}
Due to the singularity of the $A$ matrix, there may exist multiple solutions of the Lyapunov equations in \eqref{eq:LyapLike} and \eqref{eq:LyapLikeQ}. For instance, suppose a symmetric matrix $\mathcal{P}$ is a solution of \eqref{eq:LyapLike}, then any matrix $\mathcal{P} + \Delta_\mathcal{P}$, with $\Delta_\mathcal{P} = \Delta_\mathcal{P}^\top$ and $A \Delta_\mathcal{P} = 0$, is also a solution of \eqref{eq:LyapLike}. However, combining the Lyapunov equations in \eqref{eq:LyapLike} and \eqref{eq:LyapLikeQ} with the constraints in \eqref{eq:VPV} and \eqref{eq:UQU}, we can determine the pseudo Gramians $\mathcal{P}$ and $\mathcal{Q}$ uniquely.  
\begin{coro} \label{coro:GenGram}
	Let $\mathcal{P}_a$ and $\mathcal{Q}_a$ be arbitrary solutions of the Lyapunov equations in \eqref{eq:LyapLike} and \eqref{eq:LyapLikeQ}, respectively. Then, the  pseudo controllability and observability Gramians, $\mathcal{P}$ and $\mathcal{Q}$ are computed as
	\begin{subequations}
		\begin{equation} \label{eq:SolveP}
		\mathcal{P} = \mathcal{P}_a - \mathcal{J} \mathcal{P}_a \mathcal{J}^\top.
		\end{equation}
		\begin{equation} \label{eq:SolveQ}
		\mathcal{Q} = \mathcal{Q}_a - \mathcal{J}^\top \mathcal{Q}_a \mathcal{J}.
		\end{equation}
	\end{subequations}
	with $\mathcal{J}$ a constant matrix defined in \eqref{defn:PseudoGramians}.
\end{coro}
\begin{proof}
	Since both $\mathcal{P}_a$ and $\mathcal{P}$ are solutions of \eqref{eq:LyapLike}, it follows from \eqref{eq:P1P2eq} that 
	\begin{equation} 
	\begin{split}
	\mathcal{P}_a-\mathcal{P}  = \mathcal{J} (\mathcal{P}_a-\mathcal{P}) \mathcal{J}^\top  
	=  \mathcal{J} \mathcal{P}_a  \mathcal{J}^\top,
	\end{split}
	\end{equation}
	where the second equality holds due to \eqref{eq:VPV}. Thus, \eqref{eq:SolveP} is verified, and \eqref{eq:SolveQ} can be proven analogously. 
\end{proof}


Hereafter, we discuss the relation between the controllability and the observability of the semistable system $ \bm{\Sigma}$ and the pseudo Gramians. 
\begin{thm} \label{thm:controlrank}
	Consider a semistable system $\bm{\Sigma}$ with pseudo controllability and observability Gramians $\mathcal{P}$ and $\mathcal{Q}$, respectively. Let $m$ be the algebraic (or geometric) multiplicity of the zero eigenvalues of $A$. Then, 
	\begin{enumerate}
		\item $\bm{\Sigma}$ is controllable if and only if   $\mathtt{rank}(\mathcal{P})=n-m$ and $\xi^\top B \ne 0$, for any nonzero vector $\xi \in \mathcal{N}(A^\top)$;
		
		\item $\bm{\Sigma}$ is observable if and only if $\mathtt{rank}(\mathcal{Q})=n-m$ and $C\xi \ne 0$, for any nonzero vector $\xi \in \mathcal{N}(A)$.
	\end{enumerate}
\end{thm}
\begin{proof}
	We prove the first statement. Consider the \textit{finite-time controllability Gramian} of $ \bm{\Sigma}$: 
	\begin{equation} \label{def:fintimGramsP}
	{P}_s (0,t_f) =  \int_{0}^{t_f} 
	e^{A\tau}  BB^\top  e^{A^\top\tau}  \mathrm{d}\tau,
	\end{equation}
	which is bounded and positive semidefinite as $t_f$ is finite, and from \cite{antoulas2005approximation}, we have $\bm{\Sigma}$ is controllable on $\left[0, t_f\right]$ if and only if ${P}_s (0,t_f)$ in \eqref{def:fintimGramsP} is full rank. Analogously, \textit{finite-time pseudo controllability Gramian} of $\bm{\Sigma}$ is defined as
	\begin{equation} \label{eq:GenFintGramian}
	\mathcal{P}(0,t_f) =  \int_{0}^{t_f} 
	(e^{A\tau}-\mathcal{J}) BB^\top (e^{A^\top\tau}-\mathcal{J}^\top) \mathrm{d}\tau \succcurlyeq 0.
	\end{equation}
	We then prove the necessary and sufficient condition of the controllability of $\bm{\Sigma}$ using the rank of $\mathcal{P}(0,t_f)$.
	
	\textit{Necessity}: Let $\bm{\Sigma}$ be controllable. Thus,  ${P}_s (0,t_f) \succ 0$, i.e., for all nonzero vector $\xi$, 
	\begin{equation} \label{eq:xiPxi}
	\xi^\top {P}_s (0,t_f) \xi =  \int_{0}^{t_f} 
	\xi^\top e^{A\tau}  B B^\top  e^{A^\top\tau} \xi \mathrm{d}\tau > 0.
	\end{equation}
	Equivalently, there is
	no vector $\xi \neq 0$ such that $\xi^\top e^{A\tau}  B = 0$, $\forall~ \tau \in \left[0, t_f\right]$.  
	
	To determine the rank of $\mathcal{P}(0,t_f)$, we first show that 
	$
	\mathtt{dim}(\mathcal{N}(\mathcal{P}(0,t_f)))= m.
	$ 
	Consider the decomposition of $A$ in \eqref{eq:decompostion}, where 
	\begin{equation}
	\mathcal{R}(U) \cup \mathcal{R}(\bar{U}) = \mathcal{R}(U) \cup \mathcal{R}(V)^\perp = \mathbb{R}^n,
	\end{equation}  
	such that an arbitrary nonzero vector $\xi \in \mathbb{R}^n$ can be written as
	\begin{equation} \label{eq:thm:decom}
	\xi = \alpha \xi_1 + \beta \xi_2,
	\end{equation}
	where $\alpha, \beta$ are scalars, and $\xi_1 \in \mathcal{R}(V)$, $\xi_2 \in \mathcal{R}(U)^\perp$, which satisfy
	\begin{equation} \label{eq:thm:x1x2} 
	\xi_1^\top (e^{A\tau}-\mathcal{J}) B = 0, 
	\ \text{and} \
	\xi_2^\top \mathcal{J} = \xi_2^\top U V^\top = 0.
	\end{equation}
	The first equation in \eqref{eq:thm:x1x2} holds due to
	\begin{align}
		 V^\top \left(e^{A\tau}-\mathcal{J}\right) B 
		=  V^\top \left(I+\sum_{k=1}^{\infty} 
		\dfrac{A^k\tau^k}{k!} - \mathcal{J}\right) B   
		=  (V^\top  -  V^\top U V^\top)  B = 0,
		\end{align}
		where the equations $V^\top A = 0$ and $V^\top U = I_m$ are used.
	
	Any nonzero vector $\tilde{\xi} \in \mathcal{N}(\mathcal{P}(0,t_f))$ 
	is characterized by
	\begin{equation} \label{eq:halfP}
	\tilde\xi^\top (e^{A\tau}-\mathcal{J}) B = 0, \ \forall~ \tau \in \left[0, t_f\right].
	\end{equation} 
	With the decomposition of the vector $\tilde\xi$ as in \eqref{eq:thm:decom}, we rewrite \eqref{eq:halfP} as
	\begin{align} \label{eq:thm:mu2}
	\tilde\xi^\top (e^{A\tau}-\mathcal{J}) B 
	=  \alpha \tilde\xi_1^\top (e^{A\tau}-\mathcal{J}) B + \beta \tilde\xi_2^\top(e^{A\tau}-\mathcal{J}) B
	= \beta \tilde\xi_2^\top e^{A\tau} B.
	\end{align}
	Therefore, $\tilde \xi \in \mathcal{N}(\mathcal{P}(0,t_f))$ if and only if $\beta = 0$ and $\alpha \ne 0$ in \eqref{eq:thm:mu2}, namely, $\mathcal{N}(\mathcal{P}(0,t_f)) = \mathcal{R}(V)$, which yields   
	\begin{equation} \label{eq:rankn-m}
	\mathtt{rank} (\mathcal{P}(0,t_f)) = n - \mathtt{dim}(\mathcal{R}(V)) = n-m. 
	\end{equation}
	Furthermore, when $\bm{\Sigma}$ is controllable, we also obtain $\xi^\top B \ne 0$, for all nonzero vector $\xi \in \mathcal{N}(A^\top)$.  Otherwise, there will exist a nonzero vector $\xi \in \mathcal{R}(V)$ such that $\xi^\top \mathcal{J} = 0$, which implies that $\xi^\top e^{A\tau} B = \xi^\top (e^{A\tau}-\mathcal{J}) B = 0$. This contradicts \eqref{eq:xiPxi}. 
	
	\textit{Sufficiency}: Note that any nonzero vector $\xi \in \mathbb{R}^n$ can be decomposed as a linear combination of $\xi_1 \in \mathcal{R}(V)$ and $\xi_2 \in \mathcal{R}(U)^\perp$ as in \eqref{eq:thm:decom}.
	Since $\mathcal{R}(V)$ is in the nullspace of $\mathcal{P}(0,t_f)$, and $\mathtt{dim}(\mathcal{R}(V)) = m$, the rank of $\mathcal{P}(0,t_f)$ then implies that 
	\begin{equation}
	\xi_2^\top (e^{A\tau}-\mathcal{J}) B \ne 0, \ \forall~ \xi_2 \in \mathcal{R}(U)^\perp.
	\end{equation}
	It follows from \eqref{eq:thm:x1x2} that
	$
	\xi_2^\top  e^{A\tau} B \ne \xi_2^\top \mathcal{J} B = 0.
	$
	Moreover, 
	\begin{equation}\label{eq:thm:x1x22}
	\xi_1^\top  e^{A\tau} B = \xi_1^\top (e^{A\tau}-\mathcal{J} + \mathcal{J}) B = \xi_1^\top \mathcal{J}B.
	\end{equation}
	Observe that $\mathcal{J} B \ne 0$ is sufficient for $V^\top B \ne 0$. Thus, \eqref{eq:thm:x1x22} is nonzero for all $\xi_1 \in \mathcal{R}(V)$ since $V^\top \mathcal{J} B  = V^\top UV^\top B = V^\top B \ne 0$. Consequently, we obtain $\xi^\top e^{A\tau} B \ne 0$, for any nonzero vector $\xi$, i.e., ${P}_s (0,t_f)$ is positive definite. It means that $\bm{\Sigma}$ is controllable. 
	
	Finally, the first statement in the theorem is obtained as $t_f \rightarrow \infty$. The proof of the observability part follows a dual statement, Hence, the details are omitted here.  
\end{proof}

Moreover, the proposed pseudo Gramians are also relevant to the minimum input and output energy of a semistable system $\bm{\Sigma}$.
\begin{thm} \label{thm:energy}
	Consider the semistable system $\bm{\Sigma}$ and its pseudo controllability and observability Gramians $\mathcal{P}$ and $\mathcal{Q}$, respectively.
	\begin{enumerate}
		\item  If $(A, B)$ is controllable, then the least input energy required to steer the system state from $0$ to $x_0 \in \mathcal{N}(A)^\perp$ in infinite time is given by
		\begin{equation}  \label{eq:Lc}
		L_c(x_0) = \min\left\{\int_{-\infty}^{0} \| u(\tau)\|^2 \mathrm{d} \tau \right\}= x_0^\top \mathcal{P}^\dagger x_0,
		\end{equation}
		where $\mathcal{P}^\dagger$ is the pseudoinverse of $\mathcal{P}$, and $u(\tau)\in \mathcal{L}_2$, $x(-\infty)=0$, $x(0) = x_0 \in \mathcal{N}(A)^\perp$.
		
		\item If $(C, A)$ is observable, then the energy of the outputs produced by a given initial state $x_0 \in \mathcal{N}(A)^\perp$ and zero
		input is
		\begin{equation} \label{eq:Lo}
		L_o(x_0) = \int_{0}^{\infty} \| y(\tau) \|^2 \mathrm{d}\tau = x_0^\top \mathcal{Q} x_0, 
		\end{equation}
		with $x(0) = x_0 \in \mathcal{N}(A)^\perp$, $u(\tau) = 0$, $\forall~\tau \geq 0$.
	\end{enumerate} 
\end{thm}
\begin{proof}
	First, the controllability energy function $L_c(x_0)$ in \eqref{eq:Lc} is proven. Consider a coordinate transformation $z(t): = \mathcal{U}^{-1} x(t)$, with $\mathcal{U}^{-1}$ in \eqref{eq:decompostion}. Then, we obtain  
	\begin{equation} \label{eq:transf}
	\begin{bmatrix}
	\dot{z}_1(t) \\ \dot{z}_2(t)
	\end{bmatrix}
	= \begin{bmatrix}
	0 & 0 \\ 0 & \bar{A}
	\end{bmatrix}
	\begin{bmatrix}
	\dot{z}_1(t) \\ \dot{z}_2(t)
	\end{bmatrix}
	+ 
	\begin{bmatrix}
	V^\top B \\ \bar{V}^\top B
	\end{bmatrix} u(t),
	\end{equation}
	with $z_1(t) = V x(t) \in \mathbb{R}^m$, and $z_2(t) = \bar{V} x(t) \in \mathbb{R}^{n-m}$. 
	Note that the subsystem with the state $z_2(t)$ is asymptotically stable due to the Hurwitz matrix $\bar{A}$, and its controllability Gramian is given as
	\begin{equation}
	\bar{\mathcal{P}}: = \int_{0}^{\infty} e^{\bar{A} \tau} \bar{V}^\top B B^\top \bar{V} e^{\bar{A}^\top\tau} \mathrm{d} \tau = \bar{V}^\top \mathcal{P} \bar{V},
	\end{equation}
	where the latter equation is obtained by multiply $\bar{V}^\top$ and $\bar{V}$ to the left and right sides of \eqref{defn:P}, respectively.
	
	For any $x_0 \in \mathcal{N}(A)^\perp$, we have $z_1(0) = 0$ and $z_2(0) = \bar{V}^\top x_0$. Thus, the input energy $L_c(x_0)$ that required to steer $x(t)$ from $0$ to $x_0 \in \mathcal{N}(A)^\perp$ is equivalent to the energy needed to steer $z_2(t)$ the state from $z_2(-\infty) = 0$ to $z_2(0)$.  Therefore, it follows from \cite{antoulas2005approximation} that
	\begin{align} \label{eq:Lc1}
	L_c(x_0) =  z_2(0)^\top \bar{\mathcal{P}}^{-1} z_2(0) 
	=  x_0^\top \bar{V} (\bar{V}^\top \mathcal{P}\bar{V})^{-1} \bar{V}^\top x_0.
	\end{align} 
	 We then show that $\mathcal{P}^\dagger: = \bar{V} (\bar{V}^\top \mathcal{P}\bar{V})^{-1} \bar{V}^\top$ is the pseudoinverse of $\mathcal{P}$. Consider the similarity transformation in \eqref{eq:decompostion}, where $U V^\top + \bar{U} \bar{V}^\top = \mathcal{U} \mathcal{U}^{-1} = I$, and $V^\top \mathcal{P} = 0$. The following Moore-Penrose conditions are verified.
		\begin{align}
		\mathcal{P} \mathcal{P}^\dagger \mathcal{P} & = (U V^\top + \bar{U} \bar{V}^\top) \mathcal{P}\bar{V} (\bar{V}^\top \mathcal{P}\bar{V})^{-1} \bar{V}^\top \mathcal{P}  
		 = \bar{U}\bar{V}^\top \mathcal{P} = (I - UV^\top) \mathcal{P} = \mathcal{P},
		\\
		\mathcal{P}^\dagger \mathcal{P}  \mathcal{P}^\dagger & = \bar{V} (\bar{V}^\top \mathcal{P}\bar{V})^{-1} \bar{V}^\top  \mathcal{P}  \bar{V} (\bar{V}^\top \mathcal{P}\bar{V})^{-1} \bar{V}^\top = \mathcal{P}^\dagger
		\\
		(\mathcal{P}^\dagger \mathcal{P})^\top & = \mathcal{P}\bar{V} (\bar{V}^\top \mathcal{P}\bar{V})^{-1} \bar{V}^\top = \mathcal{P} \mathcal{P}^\dagger 
		\\
		(\mathcal{P} \mathcal{P}^\dagger)^\top & =  \bar{V} (\bar{V}^\top \mathcal{P}\bar{V})^{-1} \bar{V}^\top \mathcal{P} = \mathcal{P}^\dagger \mathcal{P}  
		\end{align} 
		Thus, $\mathcal{P}^\dagger$ the Moore-Penrose inverse of $\mathcal{P}$, which leads to \eqref{eq:Lc} from \eqref{eq:Lc1}. 
	
	Next, we derive the observability energy function $L_o(x_0)$ as follows. With $y(\tau) = C e^{A \tau} x_0$, we obtain
	\begin{equation}
	L_o(x_0) = \int_{0}^{\infty} x_0^\top e^{A^\top \tau} C^\top C e^{A \tau} x_0 \mathrm{d} \tau,
	\end{equation}
	which is equal to $ x_0^\top \mathcal{Q} x_0$, since $\mathcal{J} x_0 = U V^\top x_0 = 0$, for all 
	 $x_0 \in \mathcal{N}(A)^\perp = \mathcal{R}(U)^\perp$.
\end{proof}

Next, we provide a sufficient and necessary condition for the semistable system \eqref{syss} being in the $\mathcal{H}_2$  and $\mathcal{H}_\infty$ space. 
\begin{thm} \label{thm:trace}
	Consider the semistable system \eqref{syss}, and  denote $  
	\eta(s) = C\left( sI_n - A \right)^{-1}  B
	$. We have $\bm{\Sigma} \in \mathcal{H}_2$ and $\bm{\Sigma} \in \mathcal{H}_\infty$ if and only if 
	\begin{equation} \label{eq:CJB=0}
	C \mathcal{J} B = 0
	\end{equation}
	Furthermore, let \eqref{eq:CJB=0} hold, then
	\begin{equation}
	\lVert \eta(s) \rVert_{\mathcal{H}_2}^2 = \mathtt{tr} (C \mathcal{P} C^\top) = \mathtt{tr} (B^\top \mathcal{Q} B),
	\end{equation}
	and $\lVert \eta(s) \rVert_{\mathcal{H}_\infty} \leq \gamma$ if there exist $\gamma>0$ and $\mathcal{K} \succcurlyeq 0$ satisfying
	\begin{align}
	\begin{bmatrix}
	A^\top \mathcal{K} + \mathcal{K} A  & \star & \star   \\
	B^\top \mathcal{K} & -\gamma I & \star \\
	C (I-\mathcal{J})   &  0 & -\gamma  I
	\end{bmatrix}\preccurlyeq 0, 
	\label{eq:LMI}
	\\ 
	\mathcal{J}^\top \mathcal{K}  \mathcal{J} = 0.
	\label{eq:JKJ} 
	\end{align}
\end{thm}
\begin{proof}
	Consider a coordinate transformation $z(t): = \mathcal{U}^{-1} x(t)$, where $\mathcal{U}^{-1}$ is defined in \eqref{eq:decompostion}. We then obtain \eqref{eq:transf} and $y(t) = \begin{bmatrix}
	C U & C \bar{U}
	\end{bmatrix} z(t)$. 
	Thereby, the transfer function of \eqref{syss} is written as
	\begin{equation}
	\eta(s) = \bar{C}(sI - \bar{A})^{-1} \bar{B} + \dfrac{1}{s} C \mathcal{J} B,
	\end{equation}
	with $\bar C = C \bar{U}$ and $\bar B = \bar{V}^\top B$.
	Note that $C \bar{U}(sI - \bar{A})^{-1} \bar{V}^\top B$ is asymptotically stable. Thus, $\eta(s) \in \mathcal{H}_2$ (or $ \mathcal{H}_\infty$) if and only if \eqref{eq:CJB=0} holds.

	Let $g(\tau): = C e^{A\tau}  B$ be the impulse response of \eqref{syss}. It follows from \cite{antoulas2005approximation} that
	\begin{equation} \label{eq:H2int}
	\begin{split}
	\lVert \eta(s) \rVert_{\mathcal{H}_2}^2 &= \mathtt{tr} \left(\int_{0}^{\infty} g(\tau)^\top g(\tau) \mathrm{d}\tau\right),
	\end{split}
	\end{equation}
	which is well-defined if and only if $g(\tau)$ is absolutely integrable, namely, in this case,
	\begin{equation}
	\lim\limits_{ \tau \rightarrow \infty}
	g(\tau) =  C \left(\lim\limits_{ \tau \rightarrow \infty}e^{A\tau}\right)  B = C \mathcal{J} B = 0,
	\end{equation}
	which then immediately yields
	\begin{equation*}
	\mathtt{tr} (C \mathcal{P} C^\top) =  \mathtt{tr} (B^\top \mathcal{Q} B) =
	\mathtt{tr} \left(\int_{0}^{\infty} g(\tau)^\top g(\tau) \mathrm{d}\tau\right).
	\end{equation*}
	
	Next, we derive the $\mathcal{H}_\infty$ norm of \eqref{syss}. If \eqref{eq:CJB=0} is satisfied, $\| \eta(s) \|_{\mathcal{H}_\infty} = \| \bar{C}(sI - \bar{A})^{-1} \bar{B} \|_{\mathcal{H}_\infty}$. By the well-known bounded real lemma \cite{scherer2005linear}, $\| \eta(s) \|_{\mathcal{H}_\infty} \leq \gamma$, if there exists $K \succ 0$ satisfies
	\begin{equation} \label{eq:Riccati1}
	\bar{A}^\top K + K \bar A + \bar C^\top \bar C + \gamma^2 K \bar B\bar B^\top K \preccurlyeq 0.
	\end{equation}
	Note that due to \eqref{eq:JKJ}, 
	\begin{equation}
	\mathcal{U}^\top \mathcal{K} \mathcal{U} = \begin{bmatrix}
	U^\top \mathcal{K} U & U^\top \mathcal{K} \bar U \\
	\bar U^\top \mathcal{K} U & \bar U^\top \mathcal{K} \bar U
	\end{bmatrix} = 
	\begin{bmatrix}
	0 & 0 \\
	0 & \bar K
	\end{bmatrix},
	\end{equation}
	with $\bar K = \bar U^\top \mathcal{K} \bar U$.
	Moreover, we have
	\begin{equation}
	\mathcal U^{-1} (I - \mathcal{J}) \mathcal{U} =   \begin{bmatrix}
	0 & 0  \\ 0 & I_{n-m}
	\end{bmatrix}.
	\end{equation}
	Thus, the following equations hold.
	\begin{align}
	&\mathcal{U}^\top A^\top \mathcal{K}\mathcal{U} =  (\mathcal{U}^{-1}A\mathcal{U})^\top \mathcal{U}^\top\mathcal{K} \mathcal{U} = \begin{bmatrix}
	0 & 0 \\ 0 & \bar{A}^\top  \bar{K} 
	\end{bmatrix}, 
	\\
	&\mathcal{U}^\top \mathcal{K} B =  \mathcal{U}^\top \mathcal{K} \mathcal{U}  \mathcal{U}^{-1} B 
	=  
	\begin{bmatrix}
	0 \\
	\bar{K} \bar{B} 
	\end{bmatrix},	
	\\
	&	C (I - \mathcal{J}) \mathcal{U} =  C \mathcal{U}^{-\top} \mathcal{U}^{\top} (I - \mathcal{J}) \mathcal{U} 
	= \begin{bmatrix}
	0 & \bar{C}
	\end{bmatrix}.
	\end{align}
	It leads to \eqref{eq:Riccati1} if we multiply  \eqref{eq:LMI} by $\blkdiag(\mathcal{U}^\top, I, I)$ and its transpose from the left and right simultaneously.  
	Thus, $\| \eta(s) \|_{\mathcal{H}_\infty} \leq \gamma$.
\end{proof}
\begin{coro} \label{coro:LMI}
	Consider the semistable transfer function $g(s) = C (sI - A)^{-1} B + D$. If $A$ is dissipative, i.e., $A + A^\top \preccurlyeq 0$, and $(I - \bar{V}\bar{V}^\top) B = 0$ with $\bar{V}$ defined in \eqref{eq:decompostion}, then $\| g(s) \|_{\mathcal{H}_\infty} \leq \kappa$ with $\kappa$ satisfying 
	\begin{align} \label{eq:LMI2}
	\begin{bmatrix}
	A^\top + A  & \star & \star   \\
	B^\top  & -\kappa I & \star \\
	C (I-\mathcal{J})   &  D & -\kappa  I
	\end{bmatrix}\preccurlyeq 0,
	\end{align}
\end{coro}
\begin{proof}
	To characterize the $\mathcal{H}_\infty$ norm of $g(s)$ with the feedthrough term $D$, \eqref{eq:LMI} is modified, where the second entry on the bottom is replaced by $D$.
	Now, let $\mathcal{K}: = \bar{V} \bar{V}^\top$, which satisfies \eqref{eq:JKJ} due to $\bar{V}^\top U = 0$. Furthermore, we have 
	\begin{equation}
	A \mathcal{K} = \bar{U} \bar{A} \bar{V}^\top  \bar{V} \bar{V}^\top = A, \ \mathcal{K} B = \bar{V}\bar{V}^\top B = B,
	\end{equation}
	which then leads to \eqref{eq:LMI2}.
\end{proof}

Note that if $A$ is Hurwitz in \eqref{syss}, i.e., $\mathcal{J} = 0$, the $\mathcal{H}_2$ norm and $\mathcal{H}_\infty$ norm are well-defined and can be characterized by the standard Gramians and a Riccati inequality, respectively \cite{antoulas2005approximation}. However, when $A$ contains semistable eigenvalues, both characterizations are not feasible any more. In contrast, Theorem~\ref{thm:trace} can be used.

\section{Directed Network Systems \& graph clustering} \label{sec:networksystem}
In this paper, the interactions among the vertices are described by a directed graph. Thereby, this section first provides necessary preliminaries on graph theory, and then introduces the model of a directed network in the form of Laplacian dynamics. Based on graph clustering, a projection framework for network structure-preserving model reduction is proposed.

\subsection{Graph Theory}

We briefly recapitulate the definitions and fundamental results from graph theory that will be used throughout this paper. For more details, we refer to e.g., \cite{Wu2007Synchronization,Wu2005Connectivity,Agaev2005Spectra}.

A digraph graph $\mathcal{G} = \left(\mathbb{V}, \mathbb{E}\right)$ is composed by a set of vertices $\mathbb{V}: = \{1,2, \cdots, n\}$ and a set of directed edges $\mathbb{E} \subseteq \mathbb{V} \times \mathbb{V}$. Each directed edge $a_{ij} = (i,j) \in \mathbb{E}$ indicates that information flows from vertex $j$ to vertex $i$. Based on the connectedness of a digraph, the following categories are defined. 
 
\begin{defn}
	A digraph $\mathcal{G}$ is {\bfseries weakly  connected} ($\mathcal{G} \in \mathbb{G}_w$) if there exists an undirected path between any $i,j \in \mathbb{V}$. Particularly, if there exists a directed path in each direction between any $i,j \in \mathbb{V}$, $\mathcal{G}$ is {\bfseries strongly connected} ($\mathcal{G} \in \mathbb{G}_s$). Furthermore, for every pair of vertices $i,j \in \mathbb{V}$, if there exists a vertex $k \in \mathbb{V}$ that can reach $i,j$ by a directed path, $\mathcal{G}$ is {\bfseries quasi strongly connected} ($\mathcal{G} \in \mathbb{G}_q$). 
\end{defn}

\begin{defn}
	A {\bfseries strongly connected component} (SCC) of a digraph $\mathcal{G}$ is a maximal strongly connected subgraph. Any digraph $\mathcal{G}$ can be partitioned into several SCCs. If a SCC only has outflows, it is then called a {\bfseries leading strongly connected component} (LSCC) \cite{Wu2007Synchronization}. 
\end{defn}
 

A digraph $\mathcal{G} \in \mathbb{G}_w$ may contains multiple LSCCs, while  $\mathcal{G} \in \mathbb{G}_q$ only has a single LSCC. 
Generally, we have 
\begin{equation}
	 \mathbb{G}_w \supset \mathbb{G}_q \supset \mathbb{G}_s.
\end{equation}
\begin{exm}
	\begin{figure}[!tp]\centering
		\includegraphics[scale=.6]{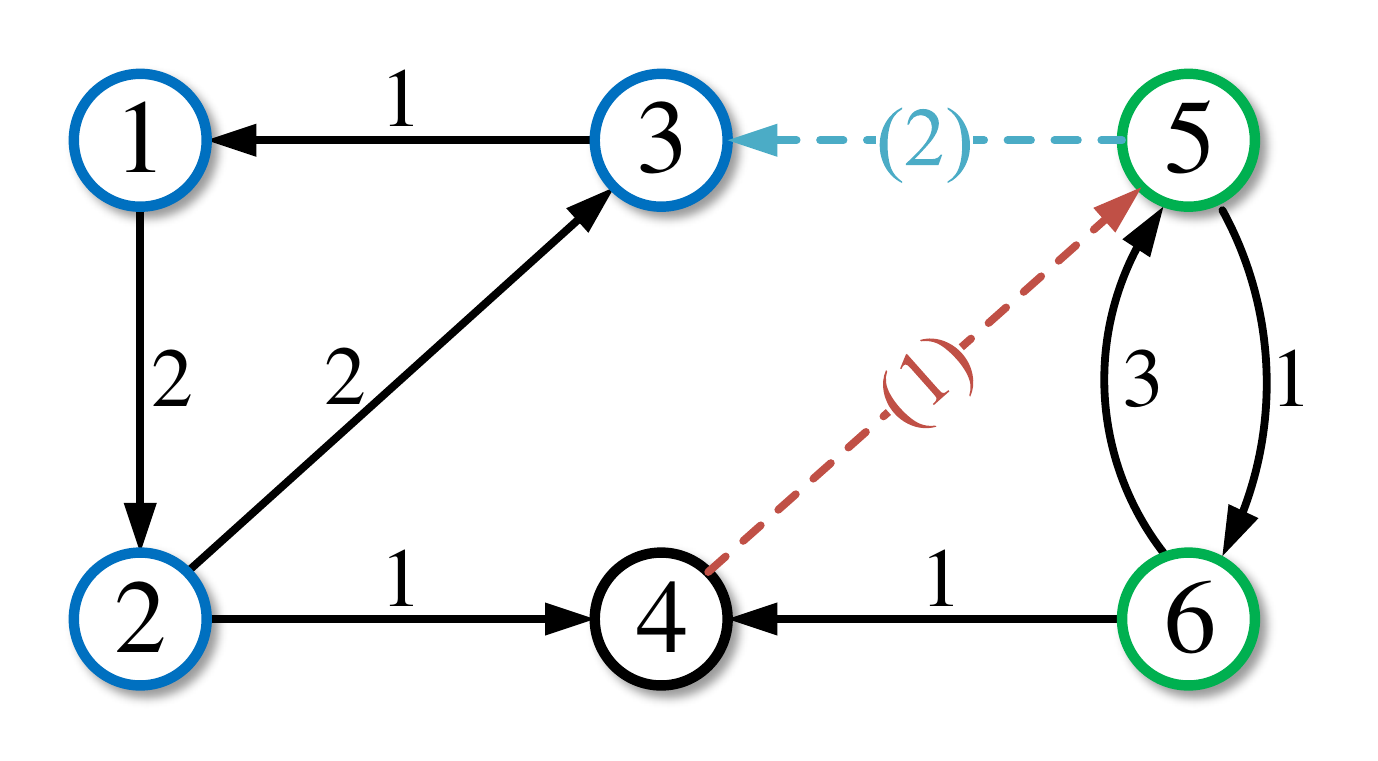}	
		\caption{Illustration of different categories of digraphs.}
		\label{fig:Ex1}
	\end{figure}
	Fig. \ref{fig:Ex1} demonstrates different types of digraphs. When only considering the edges indicated by solid arrows, $\mathcal{G} \in \mathbb{G}_w$, and there exist three SCCs in $\mathcal{G}$:
	$\{1,2,3\}$, $\{5,6\}$ and $\{4\}$, where the first two SCCs are LSCCs. Whereas, after an extra edge$a_{45}$ (dashed arrow (1)) is added, this digraph becomes quasi strongly connected, i.e., $\mathcal{G} \in \mathbb{G}_q$, which contains only two SCCs: $\{1,2,3\}$, $\{4,5,6\}$, and the first one is the LSCC. Moreover, $\mathcal{G}$ will be strongly connected, when the vertices $2$ and $4$ are also connected by $a_{24}$ represented by the dashed arrow (2).
\end{exm}
In this paper, we consider the most general case that a network is described by a weakly connected graph, and we treat the strongly connected and quasi strongly connected digraphs as special cases. 
The scenario that the digraph is disconnected, i.e., not even weakly connected, is the least interesting, since it corresponds to a set of isolated subnetworks, which can be analyzed independently.

The topology of $\mathcal{G}$ can be represented by a so-called \textit{Laplacian matrix}, which is defined as follows. Suppose $\mathcal{G}$ is edge-weighted, and denote $\mathcal{W} \in \mathbb{R}^{n \times n}$ as the \textit{weighted adjacency matrix}, whose $(i,j)$ entry, denoted by $w_{ij}$, is positive if the directed edge $a_{ij} \in \mathbb{E}$, and $w_{ij}=0$ otherwise.  
Then, the \textit{Laplacian matrix} of $\mathcal{G}$ is defined by 
\begin{equation} \label{defn:Laplacian}
	\mathcal{L} = \mathsf{diag}(\mathcal{W}\mathds{1}) - \mathcal{W},
\end{equation}
such that the $(i,j)$ entry of $\mathcal{L}$ is read as
\begin{equation} \label{defn:Laplacian1}
\mathcal{L}_{ij} = \left\{ \begin{array}{ll} 
\sum_{j=1,j\ne i}^{n} w_{ij}, & i = j,\\
-w_{ij}, & \text{otherwise.}
\end{array}
\right.
\end{equation}

The definition (\ref{defn:Laplacian}) implies that $\mathcal{L}$ at least has one zero eigenvalue, since $\mathds{1}_n \in \mathsf{ker}(\mathcal{L})$. Specifically, the algebraic (or geometric) multiplicity of the zero eigenvalues is related to the number of LSCCs in $\mathcal{G}$. 
\begin{lem} \cite{Godsil2013AlgebraicGraph,Agaev2005Spectra} \label{lem:LapSemi}
	Consider a weakly connected weighted digraph $\mathcal{G}$. The Laplacian matrix $\mathcal{L}$ has 
	semisimple eigenvalues at the origin with the multiplicity $m$, which coincides with the number of LSCCs in $\mathcal{G}$. Particularly, if $\mathcal{G} \in \mathbb{G}_s$ or $\mathcal{G} \in \mathbb{G}_q$, then $\mathcal{L}$ only has a single zero eigenvalue.
\end{lem}

\subsection{Directed Network Systems}

A directed network system describes the dynamics evolving over a digraph $\mathcal{G}$. Let $x_i(t) \in \mathbb{R}$ be the state of a vertex $i$, which is diffusively coupled with the other vertices as
\begin{equation}
\begin{split}
	\dot{x}_i(t) = - \sum_{j=1,j\ne i}^{n} w_{ij} \left[ {x}_i(t) - {x}_j(t)\right]  + \sum_{k=1}^{p} f_{ik} u_k(t), \\
\end{split}
\end{equation}
where $w_{ij} \in \mathbb{R}_+$ and $f_{ij} \in \mathbb{R}$ represent the weight of the edge $a_{ij}$ and the amplification of the input $u_k \in \mathbb{R}$ acting on the vertex $i$, respectively. 
Then, we describe the dynamics of all the vertices on a directed consensus network by the following linear time-invariant system
\begin{equation} \label{sys}
	\bm{\Sigma} : \left \{
	\begin{array}{l}
	\dot{x}(t) = -\mathcal{L} x(t) +  {F}u(t), \\
	y(t) = H x(t),
	\end{array}
	\right.
\end{equation}
where $x(t) \in \mathbb{R}^n$ is the collection of the vertex states. The vectors $u \in \mathbb{R}^p$ and $y \in \mathbb{R}^q$ are the external control flows and measurements. The input and output matrices $ {F}$ and $H$ then represent the distributions of the external inflows and outflows, respectively. The Laplacian matrix $\mathcal{L}$, following the definition in (\ref{defn:Laplacian}),  is associated with a \textit{weakly connected weighted} digraph $\mathcal{G}$, which reflects the diffusive coupling among the vertices of $\mathcal{G}$.
Throughout the paper, we
suppose the digraph $\mathcal{G}$ contains $n_d$ SCCs, $\{\mathcal{S}_1, \dots, \mathcal{S}_{n_d}\}$, such that $\mathcal{S}_i \cap \mathcal{S}_j = \varnothing$, for any $i,j$, and
$
 \cup_{j=1,\cdots,n_d} \mathcal{S}_i = \mathbb{V}
$. Moreover, the set $\mathbb{S}_L \subseteq \mathbb{V}$ collects all the vertices in the LSCCs of $\mathcal{G}$, and $\mathcal{S}_i \subseteq \mathbb{S}_L$ if $\mathcal{S}_i$ is a LSCC of $\mathcal{G}$.

A typical example of (\ref{sys}) is a chemical reaction network, (see e.g.,  \cite{Mirzaev2013LaplacianDynamics,Ahsendorf2014GeneRegulation,Gunawardena2012Model}), where chemical species are the vertex states $x_{i}(t)$. The directed edges represent a series of chemical reactions converting source species to target species, and the edge weights are the rate constants of the corresponding reactions. 

\begin{rem}
	The network system $\bm{\Sigma}$ is semistable due to the the Laplacian matrix $\mathcal{L}$, which is singular and reducible (i.e., $\mathcal{L}$ is similar
	via a permutation to a block upper triangular matrix). Furthermore, by Ger{\v{s}}gorin's circle theorem (see,
	e.g., \cite{Johnson1990Matrix}), the real part of each nonzero eigenvalue of $\mathcal L$ is strictly positive.
\end{rem}
 
A strongly connected digraph is called \textit{balanced} if the indegree and outdegree of each node are equal \cite{godsil2013algebraic,Wu2007Synchronization}. In e.g. the mass action kinetics chemical reaction networks \cite{Rao2015balancing,Horn1972action}, the balancing of a directed network is necessary to preclude sustained oscillations, multi-stability or other types of exotic dynamic behavior. This paper extends the definition of balanced digraphs to the weakly connected case.
\begin{defn} \label{defn:balancedGraph}
	A weakly connected digraph $\mathcal{G}$ is \textbf{generalized balanced} if each LSCC is balanced, i.e.,
	$
	\sum_{j \in \mathcal{S}_k} w_{ij} = \sum_{j \in \mathcal{S}_k} w_{ji}, 
	$
	$\forall i \in \mathcal{S}_k \subseteq \mathbb{S}_L$.
\end{defn}

The following lemma then shows that any directed network system in (\ref{sys}) can be converted to its generalized balanced form.

\begin{lem} 
	Consider a directed network system $\bm{\Sigma}$ in (\ref{sys}). There always exists an equivalent representation: 
	\begin{equation} \label{sysbal}
	\bm{\Sigma} : \left \{
	\begin{split} 
	M \dot{x}(t) &= - {L} x(t) + MFu(t), \\
	y(t) &= H x(t),
	\end{split}
	\right.
	\end{equation} 
	where $M \in \mathbb{R}^{n \times n}$ is positive diagonal such that $L: = M \mathcal{L}$ is a Laplacian matrix associating with a generalized balanced digraph.
\end{lem}
\begin{proof}
	A reducible Laplacian matrix $\mathcal{L}$ is permutation-similar to a block upper triangular matrix. Thus, there exists a permutation matrix $\mathcal{T}_\mu$ such that
	\begin{equation} \label{eq:Lblkdiag}
		\mathcal{L} = \mathcal{T}_\mu \cdot \begin{bmatrix}
		\mathcal{L}_{l1} & \cdots  & 0 & 0\\
		\vdots & \ddots & \vdots & \vdots\\
		0 & \cdots & \mathcal{L}_{lm} & 0\\
		* & \cdots &* & \mathcal{L}_{r}\\
		\end{bmatrix},
	\end{equation}
	where the first $m$ diagonal blocks, $\mathcal{L}_{li}$ ($i = 1, \cdots, m$), are the Laplacian matrices associated with the $m$ LSCCs of $\mathcal{G}$, while $\mathcal{L}_{r}$ relates to the remaining vertices in $\mathcal{G}$. It is indicated by \cite{Wu2007Synchronization,XiaodongCDC2017Digraph} that the Laplacian associated to a strongly connected digraph has a simple zero eigenvalue, whose associated left
	eigenvector has all positive entries. Therefore, there exists a positive vector $\nu_i$ such that $\mathcal{L}_{li}^T \nu_i = 0$ for each $i = 1,\cdots, m$. Then, it is verified that
	$\mathsf{diag}(\nu_i) \mathcal{L}_{li}$ represents a balanced directed subgraph, since 
	\begin{equation}
		\mathds{1}^T \mathsf{diag}(\nu_i) \mathcal{L}_{li} = 0 \ \text{and} \  \mathsf{diag}(\nu_i) \mathcal{L}_{li} \mathds{1}  = 0.
	\end{equation}
	Let 
	\begin{equation} \label{eq:M}
		M : = \mathsf{diag}\left([\nu_1^T, \cdots, \nu_m^T, \nu_r^T]\right) \cdot \mathcal{T}_\mu^T
	\end{equation}
	with $\nu_r$ an arbitrary positive vector. Then, $L: = M \mathcal{L}$ represents a generalized balanced digraph by Definition \ref{defn:balancedGraph}. 
\end{proof}

\begin{exm} \label{ex2}
	Consider the weakly connected graph in Fig. \ref{fig:Ex1}, where the weights the edges are labeled. The weighted Laplacian matrix is written as 	
	\begin{equation} \label{ex:L}
	\mathcal{L} = 
	\left[
		\begin{array}{ccc;{2pt/2pt}c;{2pt/2pt}cc}
		1  &   0  &  -1  &   0  &   0  &   0\\
		-2 &    2 &    0 &    0 &    0 &    0\\
		0  &  -2  &   2  &   0  &   0  &   0\\ \hdashline[2pt/2pt]
		0  &  -1  &   0  &   2  &   0  &   -1\\ \hdashline[2pt/2pt]
		0  &   0  &   0  &   0  &   3  & -3\\
		0  &   0  &   0  &   0   & -1  &   1
		\end{array} 
	\right].
	\end{equation}
	By (\ref{eq:M}), we choose 
	$M = \mathsf{diag}\left([2, 1, 1, \alpha, 1, 3]^T\right)$ with $\alpha$ an arbitrary positive scalar such that
	\begin{equation}
		L = M \mathcal{L} = 
		\left[	\begin{array}{ccc;{2pt/2pt}c;{2pt/2pt}cc}
			2  &   0  &  -2  &   0  &   0  &   0\\
			-2 &    2 &    0 &    0 &    0 &    0\\
			0  &  -2  &   2  &   0  &   0  &   0\\ \hdashline[2pt/2pt]
			0  &  -\alpha  &   0  &   2\alpha  &   0  &   -\alpha\\ \hdashline[2pt/2pt]
			0  &   0  &   0  &   0  &   3  & -3\\
			0  &   0  &   0  &   0   & -3  &   3
		\end{array} 
		\right]
	\end{equation}
	is a Laplacian matrix associating with a generalized balanced digraph.
\end{exm}

The introduction of the generalized balanced form (\ref{sysbal}) of directed networks is meaningful, as in the following sections, it will be employed to define the vertex clusterability and dissimilarity. Moreover, it can be seen as an extension of undirected networks in \cite{XiaodongECC2016}, where $M$ is the vertex weights and $L$ corresponds to an undirected graph.

\subsection{Projection by graph clustering} \label{sec:projection}

This subsection constructs the reduced network system in the Petrov-Galerkin framework in which the projection matrix are chosen as the characteristic matrix of graph clustering. Before proceeding, we provide some notions regarding to graph clustering \cite{XiaodongECC2016,Monshizadeh2014}.

Consider a weakly connected digraph $\mathcal{G} = \left( \mathbb{V}, \mathbb{E}\right)$. Then, a \textit{graph clustering} is to divide the vertex set $\mathbb{V}$ into $r$ nonempty and disjoint subsets, i.e.,  $\{  \mathcal{C}_1,\mathcal{C}_2,\cdots,\mathcal{C}_r\}$, where $\mathcal{C}_i$ is called a \textit{cell} (or a \textit{cluster}) of $\mathcal{G}$. 
\begin{defn}	
	The \textbf{characteristic matrix} of the clustering $\mathcal{S}$ is denoted by a binary matrix $\Pi \in \mathbb{R}^{n \times r}$, whose $(i,j)$-entry is defined by
	\begin{equation} \label{partition}
	\Pi_{ij} := \left\{ \begin{array}{ll}
	1, & \text{vertex $i \in \mathcal{C}_j$,}\\ 0, & \text{otherwise.}
	\end{array}
	\right.
	\end{equation}
\end{defn}
Clearly, the entries in each row of $\Pi$ has a single $1$ entry while all the others are $0$, which means that each vertex is included in a unique cell. Moreover, $\Pi$ satisfies
\begin{equation} \label{eq:Pi11}
	\Pi \mathds{1}_n = \mathds{1}_r. 
\end{equation}

Now, we consider the system $\bm{\Sigma}$ on digraph $\mathcal{G}$ of $n$ vertices. To formulate a reduced model of dimension  $r$, we first find a graph clustering that partitions the vertices of $\mathcal{G}$ into $r$ cells. Then, we use the characteristic matrix of the clustering to construct the projection in the Petrov-Galerkin framework, which yields a reduced network system. Specifically, the following left and right projection matrices are utilized:
\begin{equation} \label{projection}
\Pi^\dagger : = (\Pi^T N \Pi)^{-1} \Pi^T N, \text{and} \ \Pi,
\end{equation}
where $N$ is a diagonal matrix such that $\Pi^T N \Pi$ is invertible, and $\Pi^\dagger \in \mathbb{R}^{r \times n}$ is the \textit{reflexive generalized inverse} of $\Pi$ (see \cite{rao1971generalized} for the definition). Both $\Pi$ and $N$ are to be determined in the latter
sections. Wherein, $\Pi$ is formed as the clustering of $\mathcal{G}$ is selected. 

The reason that we select the pair of projection matrices in (\ref{projection}) is due to its potential to preserve a network structure. With this projection, the $r$-dimensional projected model is given as
\begin{equation} \label{sysr}
\bm{\hat{\Sigma}} : \left \{
\begin{array}{l}
\dot{z}(t) = - \hat{\mathcal{L}} z(t) + \hat{F} u(t), \\
\hat{y}(t) = \hat{H} z(t),
\end{array}
\right.
\end{equation}
where $z(t) \in \mathbb{R}^r$, and $$\hat{\mathcal{L}} := \Pi^\dagger \mathcal{L} \Pi, \ \hat{F}=\Pi^\dagger F \ \text{and} \ \hat{H} = H \Pi.$$ 
It can be verified that $\hat{\mathcal{L}} \mathds{1}_r = 0$, and $\hat{\mathcal{L}}$ has nonnegative diagonal entries and nonpositive off-diagonal elements. Thus, the reduced matrix $\hat{\mathcal{L}} \in \mathbb{R}^{r \times r}$ is a lower-dimensional Laplacian matrix representing a digraph with fewer vertices, and the reduced-order system $\bm{\hat{\Sigma}}$ models a smaller-sized weakly connected directed network. In other words, the network structure is guaranteed to be preserved. 
 
Under the projection framework (\ref{projection}), the rest of this paper investigate the structure-preserving model reduction problem of the system $\bm{\Sigma}$, formulated as follows.
\begin{prob} \label{prob:appx}
	Given a directed network system $\bm{\Sigma}$ in (\ref{sys}), find matrices $\Pi$ and $N$ such that the obtained reduced-order model $\bm{\hat{\Sigma}}$ in (\ref{sysr}) approximates the original system $\bm{\Sigma}$ in a way that $\lVert \bm{\Sigma}-\bm{\hat{\Sigma}} \rVert_{\mathcal{H}_2}$ is bounded and small.
\end{prob}

\section{Model Reduction}
\label{sec:Reduction}

The strategy of constructing a suitable reduced-order network system is discussed with two parts in this section. The first part shows the conditions to guarantee the boundedness of the approximation error $\lVert \bm{\Sigma}-\bm{\hat{\Sigma}} \rVert_{\mathcal{H}_2}$, and the second part develops an effective scheme to find an appropriate clustering such that the reduction error $\lVert \bm{\Sigma}-\bm{\hat{\Sigma}} \rVert_{\mathcal{H}_2}$ is small. 

\subsection{Clusterability}

Denote the transfer matrices of $\bm{\Sigma}$ and $\bm{\hat{\Sigma}}$ by
\begin{equation} \label{eq:transfun}
\eta(s) = H\left( sI_n + \mathcal{L} \right)^{-1}  F, \ \text{and} \ \hat{\eta}(s) 
= \hat{H} (sI_r + \hat{\mathcal{L}} )^{-1}  \hat{F},
\end{equation}
respectively. Note that both $\bm{\Sigma}$ and $\bm{\hat{\Sigma}}$ are not asymptotically stable, which means $\lVert \eta(s) \rVert_{\mathcal{H}_2}$ or $\lVert \hat{\eta}(s) \rVert_{\mathcal{H}_\infty}$ may be unbounded potentially. 
The results in \cite{XiaodongCDC2017Digraph,ishizaki2015clustereddirected} shows that when $\mathcal{G} \in \mathbb{G}_s$, even a random partition of $\mathbb{V}$ can deliver an bounded reduction error $\lVert \eta(s) - \hat{\eta}(s) \rVert_{\mathcal{H}_2}$ with a properly chosen $N$. However, such a conclusion no longer holds for more general digraphs $\mathcal{G} \in \mathbb{G}_q$ or $\mathcal{G} \in \mathbb{G}_w$. Thereby, the following definition is introduced based on the generalized balanced representation of $\bm{\Sigma}$ in (\ref{sysbal}).

\begin{defn} \label{defn:clusterable}
	In the network (\ref{sys}), the vertices $i$ and $j$ are \textbf{clusterable} if $\mathbf{e}_{ij} \in \mathsf{ker}(L)^\perp$ and $\mathbf{e}_{ij} \in \mathsf{ker}(L^T)^\perp$ simultaneously, where $L$ is defined in (\ref{sysbal}).
	Furthermore, the graph clustering of $\mathcal{G}$ is \textbf{proper} if the vertices in each cell are clusterable.
\end{defn}
Furthermore, the physical meaning of the clusterability is explained in the following lemma.
\begin{lem} \label{lem:clusterability}
	The vertices $i$ and $j$ are clusterable if and only if the following conditions hold:
	\begin{itemize}
		\item the vertices $i$ and $j$ reach consensus, i.e., when $u=0$, 
		\begin{equation*}
		\lim\limits_{t \rightarrow \infty}\left[ x_i(t) - x_j(t)\right] = 0,
		\end{equation*}
		for all initial conditions.
		
		\item the vertices $i$ and $j$ are either contained in the same LSCC or $i,j \in \mathbb{V}  \backslash  \mathbb{S}_L$.
	\end{itemize}  
\end{lem} 
\begin{proof}
	Consider the decomposition of $\mathcal{L}$ as in (\ref{eq:decompostion}), and we obtain 
	\begin{equation} \label{eq:UVdef}
	\mathcal{L} U = 0, \ V^T \mathcal{L} = 0, \ \text{and} \ V^T U = I_m,
	\end{equation}
	where $U, V \in \mathbb{R}^{n \times m}$ with $m$ the algebraic multiplicity of the zero eigenvalue of $\mathcal{L}$, i.e., the number of LSCCs in $\mathcal{G}$ (see Lemma \ref{lem:LapSemi}). 
	It follows from $L = M \mathcal{L}$ that 
	\begin{equation}
		\mathsf{im}(U) = \mathsf{ker} (L), \ \mathsf{im}(M^{-1}V) =  \mathsf{ker} (L^T)
	\end{equation}
	with $M$ given in (\ref{eq:M}). By Definition \ref{defn:clusterable}, the clusterability of vertices $i$ and $j$ is thus equivalent to 
	\begin{equation} \label{eq:eijUeijV}
		\mathbf{e}_{ij}^T U = 0, \ \text{and} \ \mathbf{e}_{ij}^T M^{-1} V = 0.
	\end{equation}
	Hereafter, we prove that (\ref{eq:eijUeijV}) holds if and only if the two conditions in this lemma are satisfied.
	
	Note that for any initial condition $x_0 \in \mathbb{R}^n$, the zero input response of $\bm{\Sigma}$ converge to 
	\begin{equation} \label{eq:conver}
	\lim\limits_{t \rightarrow \infty} e^{-\mathcal{L}t} x_0 = \mathcal{J} x_0 = U V^T x_0.
	\end{equation}
	Thus, vertices $i$ and $j$ reach consensus $\forall x_0$ equivalently means that the $i$-th and $j$-th rows of $U$ coincide, i.e., $\mathbf{e}_{ij}^T U = 0$. 
	Furthermore, it follows from e.g., \cite{Wu2007Synchronization} that 
	\begin{equation} \label{eq:eiV=0}
		\mathbf{e}_{i}^T V = 0, \ \forall i \in \mathbb{V}  \backslash  \mathbb{S}_L.
	\end{equation}
	Thus, by the definition of generalized balanced graph, $\mathbf{e}_{ij}^T M^{-1} V = 0$ holds if and only if vertices $i$ and $j$ belongs to the same LSCC, or  $i,j \in \mathbb{V}  \backslash  \mathbb{S}_L$ (in the latter case, $\mathbf{e}_{i}^T M^{-1} V = \mathbf{e}_{j}^T M^{-1} V = 0$ due to (\ref{eq:eiV=0})).  
\end{proof}

\begin{rem}
	The clusterability of different types of digraphs are discussed.
	If $\mathcal{G} \in \mathbb{G}_s$, we have $U = \frac{1}{\sqrt{n}}\mathds{1}_n$ meaning that all the vertices achieve a global consensus, i.e., $\forall i,j \in \mathbb{V}$, $x_i(t) \rightarrow x_j(t)$ as $t\rightarrow \infty$. Moreover, $V \in \mathbb{R}^n$ has all positive entries \cite{XiaodongCDC2017Digraph,ishizaki2015clustereddirected} such that $\mathds{1}^T L = \mathds{1}^T M \mathcal{L} =0$. Thus, all the vertices are clusterable. When $\mathcal{G} \in \mathbb{G}_q$, i.e., $\mathcal{G}$ contains a single LSCC, the directed network can still reach a global consensus due to $\mathsf{ker} (\mathcal{L}) = \mathsf{im} (\mathds{1}_n)$, whereas there will be two sets of clusterable vertices, which are the vertices inside $\mathbb{S}_L$ and all the other vertices outside $\mathbb{S}_L$. 
	In a more general case that $\mathcal{G} \in \mathbb{G}_w$, the system $\bm{\Sigma}$ in (\ref{sys}) may \textit{not} achieve a global consensus. Instead, local consensus is achievable among the vertices that are able to influence each other. Namely, $\bm{\Sigma}$ forms cells of consensus with different consensus values in each cell. It is guaranteed that the vertices in the same LSCC are clusterable.
\end{rem}

The definition of clusterability is nontrivial as it determines the feasibility of a graph clustering. More specifically, we find that the boundedness of the approximation error between the original and reduced systems is guaranteed only if clusterable vertices are classified in the same cell. 

Consider the permutation transformation in (\ref{eq:Lblkdiag}), and denote the following set of diagonal matrices: 
\begin{equation} \label{eq:N}
\begin{split}
	\mathbb{N} := 
	&\left\{ N \in \mathbb{R}^{n \times n}:
	\Pi^T N \Pi \ \text{is invertible}, \right.\\
	& \hspace{20pt} \left. N\mathds{1} \in \mathsf{im}
	\left(  \mathcal{T}_\mu \cdot \mathsf{blkdiag}(\nu_1, \cdots, \nu_m, I)\right) \right\},
	\end{split}
\end{equation}
where $\mathcal{T}_\mu$ is the permutation matrix reforming $\mathcal{L}$ into a block upper triangular form, and $\nu_i$ is the left eigenvector of the diagonal block matrix $\mathcal{L}_{li}$ in (\ref{eq:Lblkdiag}), i.e., $\mathcal{L}_{li}^T \nu_i = 0$. Then, the clustering-based projection matrices in (\ref{projection}), i.e., $N$ and $\Pi$, are selected according to the following theorem to guarantee the boundedness of the error between $\bm{\Sigma}$ and $\bm{\hat{\Sigma}}$.
\begin{thm} \label{thm:errbounded}
	Consider the directed network system $\bm{{\Sigma}}$ in (\ref{sys}) and its reduced-order model $\bm{\hat{\Sigma}}$ in (\ref{sysr}). For all input and output matrices $H$ and $F$, the error $\lVert \eta(s) - \hat{\eta}(s) \rVert_{\mathcal{H}_2}$ is bounded if and only if $\Pi$ in (\ref{projection}) characterizes a proper clustering of $\mathcal{G}$ and $N \in \mathbb{N}$. 
\end{thm}
\begin{proof}
	The $\mathcal{H}_2$-norm of the approximation error is given by 
	\begin{equation} \label{eq:H2integral}
	\lVert \eta(s)- \hat{\eta}(s) \rVert_{\mathcal{H}_2}^2
	= \int_{0}^{\infty} \lVert \xi(t) - \hat{\xi}(t)\rVert_2^2 dt,
	\end{equation}
	where $\xi(t): = H e^{-\mathcal{L}t} F$ and $\hat{\xi}(t): = H\Pi e^{-\hat{\mathcal{L}}t} \Pi^\dagger F$ are the impulse responses of $\bm{\Sigma}$ and $\bm{\hat{\Sigma}}$, respectively.
	Since both $\xi(t)$ and $\hat{\xi}(t)$ are smooth functions over $t \in \mathbb{R}_+$, the integral in (\ref{eq:H2integral}) is finite if and only if the error $\xi(t) - \hat{\xi}(t)$ exponentially converges to zero. Hence, for general $H$ and $F$ matrices, the boundedness of $\lVert \eta(s)- \hat{\eta}(s) \rVert_{\mathcal{H}_2}$ is equivalent to
	\begin{equation} \label{eq:thmkey}
	\mathcal{J} = \Pi \hat{\mathcal{J}} \Pi^\dagger.
	\end{equation}
	with $\mathcal{J}: = \lim\limits_{\tau \rightarrow \infty} e^{-\mathcal{L}\tau}$ and $\hat{\mathcal{J}}: = \lim\limits_{\tau \rightarrow \infty} e^{-\hat{\mathcal{L}}\tau}$ .
	
	To prove the ``if'' part, we assume $\{  \mathcal{C}_1,\mathcal{C}_2,\cdots,\mathcal{C}_r\}$ to be a proper clustering of $\mathcal{G}$. With $N \in \mathbb{N}$, we verify 
	\begin{equation} \label{eq:UVprop}
	U = \Pi \Pi^\dagger U, \ \text{and} \ V^T = V^T \Pi \Pi^\dagger 
	\end{equation}
	as follows.
	Without loss of generality, assume that 
	\begin{equation} \label{eq:Piblock}
	\Pi = \mathsf{blkdiag} \left( \mathds{1}_{|\mathcal{C}_1|}, \mathds{1}_{|\mathcal{C}_2|},\cdots,\mathds{1}_{|\mathcal{C}_r|} \right).
	\end{equation}
	Accordingly, the matrices $U$ and $V$ in (\ref{eq:UVdef}) are partitioned  as 
	$U^T = [{U}_1^T, \cdots, {U}_r^T]$ and $V^T = [{V}_1^T, \cdots, {V}_r^T]$. Meanwhile, 
	the projection $\Gamma = \Pi \Pi^\dagger$ is written in a block diagonal form with the $i$-th diagonal entry as  
	\begin{equation} \label{eq:blockGamma}
	\Gamma_i =  \mathds{1}_{|\mathcal{C}_i|}(\mathds{1}_{|\mathcal{C}_i|}^T N_i \mathds{1}_{|\mathcal{C}_i|})^{-1}\mathds{1}_{|\mathcal{C}_i|}^T N_i.
	\end{equation}
	Since $N_i$ is the corresponding principal submatrix in $N$, i.e., it is diagonal and nonsingular, the equations in (\ref{eq:UVprop}) hold if and only if 
	\begin{equation} \label{eq:UViprop}
		{U}_i = \Gamma_i {U}_i, \ \text{and} \ {V}_i^T = {V}_i^T \Gamma_i.
	\end{equation}
	It follows from Lemma \ref{lem:clusterability} that ${U}_i = \mathds{1}_{|\mathcal{C}_i|}$, which verifies the first equation in (\ref{eq:UViprop}). 
	Moreover, as the vertices in $\mathcal{C}_i$ are clusterable, Lemma \ref{lem:clusterability} implies that these vertices are either contained in the same LSCC or in the set $\mathbb{V}  \backslash  \mathbb{S}_L$. In the first case, ${V}_i = N_i\mathds{1}_{|\mathcal{C}_i|}$ owing to $\delta(N) \in \mathbb{N}$, and the second case indicates that ${V}_i = 0$ from (\ref{eq:eiV=0}).
 	It is verified that the second equation in (\ref{eq:UViprop}) are satisfied in both cases. Thus, the equations in (\ref{eq:UVprop}) hold.
	Based on this, we compute $\hat{\mathcal{J}}$ in (\ref{eq:thmkey}) for the reduced-order system $\bm{\hat{\Sigma}}$. Let $\hat{U}: = \Pi^\dagger U$ and $\hat{V}^T: = V^T \Pi$, which leads to  
	\begin{equation}
	\begin{split}
	\hat{\mathcal{L}} \hat{U} &= \Pi^\dagger \mathcal{L} \Pi \Pi^\dagger U = \Pi^\dagger \mathcal{L} U = 0, \\
	\hat{V}^T \hat{\mathcal{L}} &= V^T \Pi  \Pi^\dagger \mathcal{L} \Pi = V^T \mathcal{L} \Pi = 0.
	\end{split}
	\end{equation}
	Furthermore, due to 
	\begin{equation}
	\hat{V}^T \hat{U} = V^T \Pi \Pi^\dagger U = V^T U = I_m,
	\end{equation} 
	we obtain
	\begin{equation} \label{eq:Jhat}
		\hat{\mathcal{J}}: =  \hat{U} \hat{V}^T = \Pi^{\dagger} U V^T \Pi,
	\end{equation}
	and thus,
	\begin{equation}
	\Pi \hat{\mathcal{J}} \Pi^\dagger = \Pi\Pi^{\dagger} U V^T \Pi\Pi^\dagger = UV^T = \mathcal{J}.
	\end{equation}
	Consequently, the error $\lVert \eta(s) - \hat{\eta}(s) \rVert_{\mathcal{H}_2}$ is bounded.
	
	For the ``only if'' part, $\lVert \eta(s) - \hat{\eta}(s) \rVert_{\mathcal{H}_2}$ is assumed to be bounded for all $H$ and $F$ matrices, equivalently, (\ref{eq:thmkey}) holds. Similarly, the block diagonal structure of $\Pi$ in (\ref{eq:Piblock}) is assumed without loss of generality such that (\ref{eq:thmkey}) is presented as
	\begin{equation} \label{eq:thmkey1} 
	\begin{bmatrix}
	{U}_1 \\
	\vdots \\
	{U}_r \\
	\end{bmatrix}
	\left[ 
	{V}_1^T,  
	\cdots,  
	{V}_r^T
	\right]
	=
	\begin{bmatrix}
	\mathds{1}_{|\mathcal{C}_1|}\tilde{U}_1 \\
	\vdots \\
	\mathds{1}_{|\mathcal{C}_r|}\tilde{U}_r \\
	\end{bmatrix}
	\left[ 
	\tilde{V}_1^T \Pi^\dagger_1,  
	\cdots,  
	\tilde{V}_r^T \Pi^\dagger_r
	\right],
	\end{equation}
	where $\Pi^\dagger_i : = (\mathds{1}_{|\mathcal{C}_i|}^T N_i \mathds{1}_{|\mathcal{C}_i|})^{-1}\mathds{1}_{|\mathcal{C}_i|}^T N_i$,  $\tilde{U}_i: = \mathbf{e}^T_i \tilde{U}$ and $\tilde{V}_i: = \mathbf{e}^T_i \tilde{V}$, with $\tilde{U}$ and $\tilde{V}$ fulfilling
	\begin{equation}
		\mathsf{im}(\tilde U) = \mathsf{ker}(\hat{\mathcal L}), \ \mathsf{im}(\tilde V) = \mathsf{ker}(\hat{\mathcal L}^T),
		\ \text{and} \
		{\tilde V}^T {\tilde U} = I.
	\end{equation}
	The matrices $U_i,V_i \in \mathbb{R}^{|\mathcal{C}_i| \times m}$ are the corresponding submatrices of $U$ and $V$, respectively. Then, (\ref{eq:thmkey1}) yields
	\begin{equation} \label{eq:thmkey2} 
	\begin{split}
	U_i V_j^T &= \mathds{1}_{|\mathcal{C}_i|} \tilde{U}_i \tilde{V}_j^T \Pi^\dagger_j \\
	& = \alpha_{ij} \cdot \mathds{1}_{|\mathcal{C}_i|} \mathds{1}_{|\mathcal{C}_j|}^T N_j,
	\ \ \forall i,j = 1,2,\cdots,r.
	\end{split}
	\end{equation}
	with a scalar $\alpha_{ij}: = \tilde{U}_i \tilde{V}_j^T \cdot (\mathds{1}_{|\mathcal{C}_j|}^T N_j \mathds{1}_{|\mathcal{C}_j|})^{-1}$. It follows that 
	\begin{equation}
		\mathbf{e}^T_{ij} {U}_k = 0, \ \text{and} \ V_k^T N_k^{-1} \mathbf{e}_{ij} = 0, \ \forall i,j \in \mathcal{C}_k.
	\end{equation}
	Note that $L U = 0$ and $V^T N^{-1} L = 0$ with $L$ defined in (\ref{sysbal}). Thus, we obtain $\forall i,j \in \mathcal{C}_k$,
	$\mathbf{e}_{ij} \in \mathsf{ker}(L)^\perp$ and $\mathbf{e}_{ij} \in \mathsf{ker}(L^T)^\perp$. As the result holds for all cells $\mathcal{C}_1,\mathcal{C}_2,\cdots,\mathcal{C}_r$, the graph clustering is proper by Definition \ref{defn:clusterable}. 
	
	Next, we prove that $N \in \mathbb{N}$ is also necessary for a bounded approximation error. Clearly, $\Pi^T N \Pi$ has to be invertible in (\ref{projection}). 
	A proper clustering $\{  \mathcal{C}_1,\mathcal{C}_2,\cdots,\mathcal{C}_r\}$ means that each cell $\mathcal{C}_k$ ($k = 1, \cdots, r$) is included in either $\mathbb{V}  \backslash  \mathbb{S}_L$ or a LSCC, denoted by $\mathcal{S}_\mu$. If $\mathcal{C}_k \subseteq \mathbb{V}  \backslash  \mathbb{S}_L$, we have $V_k = 0$ owing to (\ref{eq:eiV=0}) such that (\ref{eq:thmkey2}) is equal to zero, and thus $N_k$ can be an arbitrary nonsingular diagonal matrix, that is $N_k \in \mathsf{im}(I)$.
	
	For the second case that $\mathcal{C}_k \subseteq \mathcal{S}_\mu$, we can assume, without loss of generality, that the set $\mathcal{S}_\mu$ is the union of $\mu$ disjoint cells $\mathcal{C}_1, \cdots, \mathcal{C}_{k}, \cdots, \mathcal{C}_{\mu}$. 
	Let $\mathcal{L}_\mu$ be the Laplacian matrix associated with $\mathcal{S}_\mu$ and $U_s, V_s \in \mathbb{R}^{|\mathcal{S}_\mu| \times m}$ be the rows of $U$, $V$ that correspond to $\mathcal{S}_\mu$. Thus, we obtain from \cite{Wu2007Synchronization} that
	\begin{equation}
	\begin{split}
		\mathsf{im}(U_s) &= \mathsf{ker}({\mathcal L}_\mu) =  \mathsf{im}(\mathds{1}_{|\mathcal{S}_\mu|}), \\ \mathsf{im}(V_s) &= \mathsf{ker}({\mathcal L}_\mu^T) = \mathsf{im}(\nu_s),
	\end{split}
	\end{equation}
	where $\nu_s$ has strictly positive entries, and
	$\nu_s \in \{\nu_1, \cdots, \nu_m \}$. Then, from (\ref{eq:thmkey2}), we can find a permutation matrix $\mathcal{T}_s$ such that
	\begin{equation*}
	\begin{split}
		\mathcal{T}_s U_s V_s^T & = \beta  \mathds{1}_{|\mathcal{S}_\mu|} \nu_s 
		\\
		&
		 = \begin{bmatrix}
			\alpha_{11}  \mathds{1}_{|\mathcal{C}_1|} \mathds{1}_{|\mathcal{C}_1|}^T N_1 & \cdots &		\alpha_{1\mu}  \mathds{1}_{|\mathcal{C}_1|} \mathds{1}_{|\mathcal{C}_\mu|}^T N_\mu
			\\
			\vdots & \ddots & \vdots
			\\
			\alpha_{\mu 1}  \mathds{1}_{|\mathcal{C}_\mu|} \mathds{1}_{|\mathcal{C}_1|}^T N_1 & \cdots &		\alpha_{\mu \mu}  \mathds{1}_{|\mathcal{C}_\mu|} \mathds{1}_{|\mathcal{C}_\mu|}^T N_\mu^T
		\end{bmatrix},
	\end{split}
	\end{equation*}
	with $\beta$ and $\alpha_{ij}$ scalars. It follows that
	$\alpha_{1 i} = \alpha_{2 i} = \cdots = \alpha_{\mu i}$, $\forall i = 1, \cdots, \mu$, and 
	\begin{equation}
		\mathsf{blkdiag} (N_1, \cdots, N_\mu) \mathds{1}_{|\mathcal{S}_\mu|} \in \mathsf{im}(\nu_s).
	\end{equation}
 	The above reasoning can be applied to all the LSCCs such that $N \in \mathbb{N}$ is obtained.
	
	That completes the proof.	
\end{proof}

Let $n_c$ be the number of the maximal clusterable cells, then Theorem \ref{thm:errbounded} implies that the reduction order $r$ should not be less than $n_c$. Otherwise, the approximation error will be unbounded.
Besides, we use $\Pi^\dagger = (\Pi^T M \Pi)^{-1}\Pi^T M $ as the left projection matrix with $M$ in (\ref{eq:M}), where we choose $\nu_r = \mathds{1}$ in this paper such that $M \in \mathbb{N}$. Thereby, the following section will focus on finding an appropriate clustering such that the approximation error $\lVert \bm{\Sigma} - \bm{\hat{\Sigma}}\rVert_{\mathcal{H}_2}$ is as small as possible.

\begin{exm}
	Consider the weakly connected graph in Fig. \ref{fig:Ex1}, whose Laplacian matrix is given in (\ref{ex:L}). Then, the right and left nullspaces can be characterized by
	\begin{equation*}
	\begin{split}
	U^T &= \begin{bmatrix}
		   0.25  & 0.25 & 0.25 & 0.125 & 0 & 0 \\
	       0 & 0 & 0 & 0.125 & 0.25 & 0.25
	\end{bmatrix},
    \\
	V^T &= \begin{bmatrix}
		 2 & 1 & 1 & 0 & 0 & 0\\
		 0 &0 &0 &0& 1& 3
	\end{bmatrix}.
	\end{split}
	\end{equation*}
	Thus, we can choose
	$
		M = \mathsf{diag}(2,1,1,1,1,3)
	$
	such that a Laplacian matrix in the generalized balanced form (\ref{sysbal}) is obtained. By Definition \ref{defn:clusterable}, a proper clustering of the digraph is given by $\mathcal{C}_1 = \{1,2,3\}$, $\mathcal{C}_2 = \{4\}$, and $\mathcal{C}_3 = \{5,6\}$, which yields
	\begin{equation}
	\Pi = \begin{bmatrix}
	1 & 1 & 1 & 0 & 0 & 0\\
	0 & 0 & 0 & 1 & 0 & 0\\
	0 & 0 & 0 & 0 & 1 & 1
	\end{bmatrix}.
	\end{equation}
	As a result, a 3-dimensional network system is obtained in form of (\ref{sysr}) with the reduced Laplacian matrix as
	\begin{equation}
	\hat{\mathcal{L}} = \Pi^\dagger \mathcal{L} \Pi
	= \begin{bmatrix}
	0   &   0  &        0\\
	-1  &   2 &   -1\\
	0  &  0 &    0
	\end{bmatrix},
	\end{equation}
	which represents a simpler weakly connected digraph as shown in Fig. \ref{fig:Ex3red1}.
	Next, we suppose $H = F = I_6$ in (\ref{sys}) and compute the approximation error: $\lVert \bm{\Sigma} - \bm{\hat{\Sigma}} \rVert_{\mathcal{H}_2} = 0.7852$, which is bounded. 
	
	For comparison reasons, we consider an alternative clustering, namely, $\mathcal{C}_1 = \{1,2 \}$, $\mathcal{C}_2 = \{3,4\}$, and $\mathcal{C}_3 = \{5,6\}$, while keeping the same $M$ matrix. Then, a different reduced-order system $\bm{\hat{\Sigma}}$ is obtained with the Laplacian matrix
	\begin{equation}
		\hat{\mathcal{L}} = \Pi^\dagger \mathcal{L} \Pi
		= \begin{bmatrix}
		\frac{2}{3}   &   -\frac{2}{3}  &        0\\
		-\frac{3}{2}  &  2 &   -\frac{1}{2}\\
		0  &  0 &    0
		\end{bmatrix}. 
	\end{equation}
	It represents a reduced digraph as in Fig. \ref{fig:Ex3red2}, which is quasi strongly connected, and the approximation error is shown to be unbounded. Next, we use the proper clustering as before but a different $M$ matrix, e.g., $M = I$ and $M = \mathsf{diag}(1,2,2,1,1,2)$. We find that both yield unbounded reduction errors.

 	\begin{figure}[!tp]\centering
 	\begin{minipage}[t]{0.5\linewidth}
 		\centering
 		\includegraphics[width=0.8\textwidth]{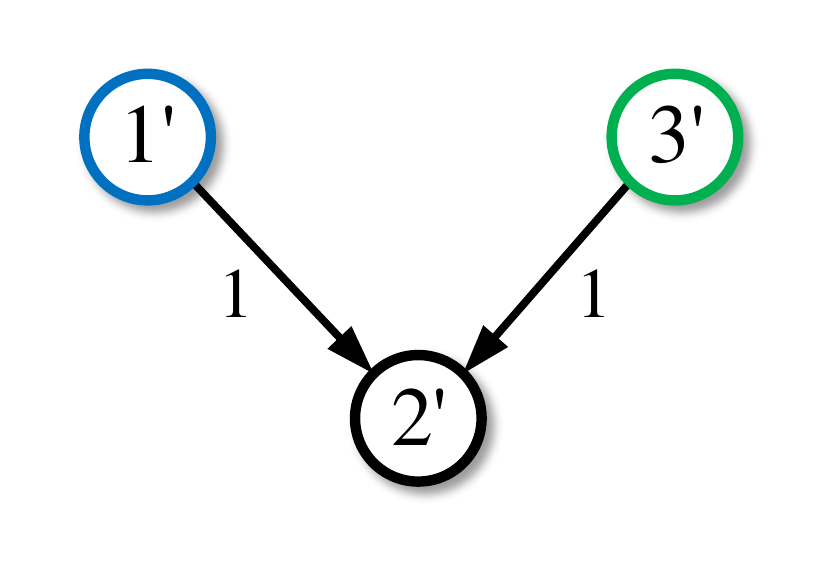}
 		\subcaption{}
 		\label{fig:Ex3red1}
 	\end{minipage}%
 	\begin{minipage}[t]{0.5\linewidth}
 		\centering
 		\includegraphics[width=0.8\textwidth]{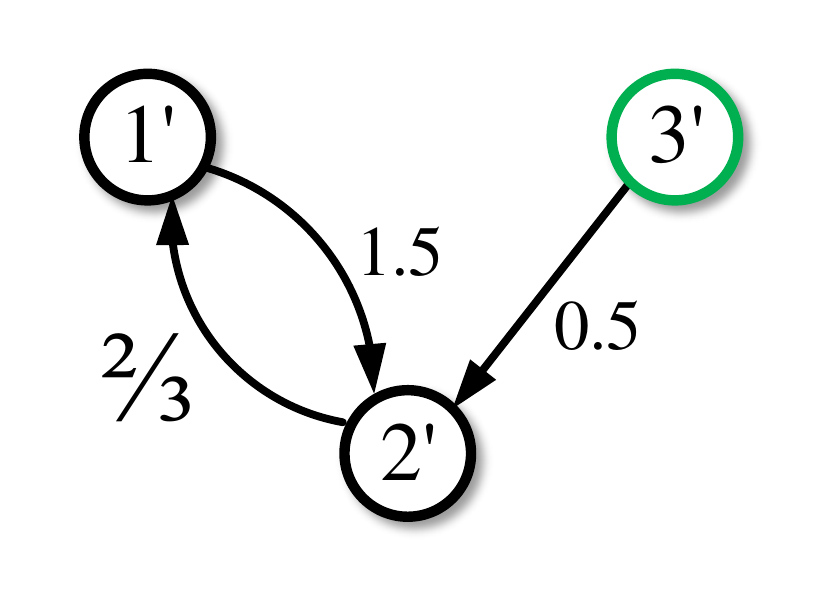}
 		\subcaption{}
 		\label{fig:Ex3red2}
 	\end{minipage}
 	\caption{(a) A reduced digraph obtained by a proper graph clustering. (b) A reduced digraph generated from an alternative graph clustering. }
 	\end{figure}
\end{exm}

%

\subsection{Vertex Dissimilarity}

The concept of clusterability determines the boundedness of the approximation error and is only dependent on the topology of the underlying network, i.e., the directed Laplacian matrix, whereas this section
investigates how to find a reduced-order model such that the magnitude of the approximation error is small. To this end, the structures of the inputs and outputs are considered as well to define the concept of vertex dissimilarity that will be regarded as a criteria for selecting appropriate graph clusterings. 
\begin{defn} 
	Consider the directed network system $\bm{\Sigma}$ in (\ref{sys}) and its generalized balanced form (\ref{sysbal}). The \textbf{dissimilarity} of a pair of clusterable vertices is defined by
	\begin{equation}
		\mathcal{D}_{ij} = \mathcal{D}^I_{ij} \cdot \mathcal{D}^O_{ij},
	\end{equation}	
	where $\mathcal{D}^I_{ij}$ and $\mathcal{D}^O_{ij}$ are input and output dissimilarities of vertices $i,j$:
	\begin{equation}
	\begin{split}
	\mathcal{D}^I_{ij}: &
	 = \lVert (\mathbf{e}_i-\mathbf{e}_j)^T (sM+L)^{-1} MF \rVert_{\mathcal{H}_2}, \\
	 \mathcal{D}^O_{ij} :&
	 = \lVert H (sM+L)^{-1} (\mathbf{e}_i-\mathbf{e}_j) \rVert_{\mathcal{H}_2}.
	\end{split}
	\end{equation}
\end{defn}

The physical meanings of the input and output dissimilarities are explained.
First, to quantify the {input dissimilarity} between a pair of vertices $i$ and $j$, we steer the network system $\bm{\Sigma}$ by injecting a series of impulse signals, i.e., let $u(t)$ in (\ref{sys}) be a vector of delta functions. Then, the responses of the vertices $i$ and $j$ are compared, i.e., 
\begin{equation}
	\Xi_i(t) - \Xi_j(t) = \left[
	\xi_{i1}(t)-\xi_{j1}(t), \cdots, \xi_{ip}(t)-\xi_{jp}(t)
	\right],
\end{equation}
with $\xi_{ij}(t) \in \mathbb{R}$ the trajectory of the vertex $i$ to the impulse in the $j$-th input channel. Thus, a proper measurement of the input dissimilarity between the vertices $i$ and $j$ is given by
\begin{equation}
	\mathcal{D}^I_{ij}= \sqrt{\int_{0}^{\infty} \left[\Xi_i(t) - \Xi_j(t)\right] \left[\Xi_i(t) - \Xi_j(t)\right]^T dt}.
\end{equation}
Essentially, the value of $\mathcal{D}^I_{ij}$ indicates how controllable the error between vertices $i$ and $j$ is. More precisely, the smaller $\mathcal{D}^I_{ij}$ means a smaller amount of input energy required to steer the vertices $i$ and $j$ to  consensus. Similarly, the {output dissimilarity} of the clusterable vertices $i$ and $j$ is characterized by
\begin{equation}
	\mathcal{D}^O_{ij}= \sqrt{\int_{0}^{\infty} \left[\Psi_i(t) - \Psi_j(t)\right]^T \left[\Psi_i(t) - \Psi_j(t)\right] dt},
\end{equation}
where $\Psi_i(t) - \Psi_j(t) \in \mathbb{R}^q$ is the output of $\bm{\Sigma}$ when only vertices $i,j$ are perturbed by an impulse signal. 
Thus, $\mathcal{D}^O_{ij}$ indicates how observable the error between vertices $i$ and $j$ is, i.e., the larger $\mathcal{D}^O_{ij}$, the less energy generated by the natural response of the perturbation on vertices $i$ and $j$, and thus more difficult the error between vertices $i$ and $j$ can be measured. 

\begin{rem}
	For vertices that are not clusterable, their dissimilarities are not properly defined, since Theorem \ref{thm:errbounded} implies that merging unclusterable vertices may cause an unbounded reduction error. Therefore, we can simply assign $\mathcal{D}_{ij}$ to be a sufficiently large positive value when $i$ and $j$ are unclusterable.
\end{rem}

Unlike the previous works in \cite{XiaodongECC2016,ishizaki2015clustereddirected} etc., the outputs of the system $\bm{\Sigma}$ are also taken into account when defining the dissimilarity. As we find that the output structures also have an impact on the approximation error, see the numerical examples in Section \ref{sec:example}, which motivates us to combine the input and output dissimilarities in order to obtain more accurate reduced networks.




Notice that the computation of $\mathcal{D}_{ij}$ for each pair of clusterable vertices $i,j$ using the definition of norms may be a formidable task when the scale of the directed network is large. Thereby, the notions of pseudo controllability and observability Gramians in Section \ref{sec:Gramian} are applied to facilitate the computation of vertex dissimilarities. 
\begin{thm} \label{thm:normcompute}
	Consider a connected directed network $\bm{\Sigma}$. The dissimilarity between two clusterable vertices $i$ and $j$ is computed as
	\begin{equation} \label{eq:Dij}
	\mathcal{D}_{ij} = 
	\sqrt{\mathbf{e}_{ij}^T\mathcal{P}\mathbf{e}_{ij}\mathbf{e}_{ij}^T M^{-1} \mathcal{Q} M^{-1}\mathbf{e}_{ij}},
	\end{equation}
	where $\mathbf{e}_{ij}: = \mathbf{e}_{i} - \mathbf{e}_{j}$, and $\mathcal{P}$, $\mathcal{Q}$ are the pseudo controllability and observability Gramians of $\bm{\Sigma}$, respectively.
\end{thm}

\begin{proof}
	Using the clusterability of  vertices $i$ and $j$, we have
	\begin{equation}
		\mathbf{e}_{ij}^T \mathcal{J} F = 0, \ \text{and} \	H \mathcal{J} M^{-1} \mathbf{e}_{ij} = 0
	\end{equation}
	Therefore, from Lemma \ref{lem:trace}, $\mathcal{D}_{ij}$ is bounded for any clusterable vertices $i,j$, and
	\begin{equation} 
	 \mathcal{D}^I_{ij}  = \sqrt{\mathbf{e}_{ij}^T\mathcal{P}\mathbf{e}_{ij}},
	\ \text{and} \
	 \mathcal{D}^O_{ij}  = \sqrt{\mathbf{e}_{ij}^T M^{-1} \mathcal{Q}M^{-1}\mathbf{e}_{ij}},
	\end{equation} 
	which gives (\ref{eq:Dij}).
\end{proof}


\subsection{Minimal Network Realization}

We discuss a novel concept of minimal realization for network systems. Before proceeding, the following definitions are introduced.

Let $\mathbb{V}_\mathcal{I}$ and $\mathbb{V}_\mathcal{O}$ be the subsets of $\mathbb{V}$ such that $i \in \mathbb{V}_\mathcal{I}$ and $i \in \mathbb{V}_\mathcal{O}$ if the vertex $i$ is steered directly by the input $u$ and directly measured by the output $y$, respectively.
\begin{defn} \label{defn:reachable}
	A vertex $j$ is \textbf{reachable} if there exists a directed path from any vertex $i \in \mathbb{V}_\mathcal{I}$ to $j$, and \textbf{detectable} if there is a directed path from $j$ to all vertex $i \in \mathbb{V}_\mathcal{O}$. Two vertices $i$ and $j$ are \textbf{0-dissimilar} if $\mathcal{D}_{ij} = 0$.
\end{defn}

Physically, unreachable vertices cannot receive information from the inputs, and thus their states is not controllable. Similarly, vertices that are undetectable cannot pass their information to the outputs, namely, their states are unobservable.  
The 0-dissimilar condition relies on the specific structures of a network topology. For some networks with special topologies, e.g., complete graphs, lattices, trees, rings, star graphs, etc., we can obtain them directly from the underlying graphs.
Consider a digraph graph $\mathcal{G} = \left(\mathbb{V}, \mathbb{E}\right)$. 
An \textit{in-neighbor} (resp. \textit{out-neighbor} ) of a vertex $i$ of $\mathcal{G}$
is a vertex $j$ with $a_{ji} \in \mathbb{E}$ (resp. $a_{ij} \in \mathbb{E}$).
The set of all the in-neighbors and out-neighbors of $i$ in $\mathcal{G}$, denoted by $\mathcal{N}_I(i)$ and $\mathcal{N}_O(i)$, are called the \textit{in-neighborhood}
and \textit{out-neighborhood} of $i$, respectively. 
Generally, if two vertices $i,j \notin \mathbb{V}_\mathcal{I}$ has a zero input dissimilarity, i.e., $\mathcal{D}^I_{ij} = 0$ if  they share the same in-neighborhood, and
\begin{equation*}
	w_{ki} = w_{kj}, \ \forall k \in \mathcal{N}_I(i) = \mathcal{N}_I(j). 
\end{equation*}
Analogously, $\mathcal{D}^O_{ij} = 0$ for two vertices $i,j \notin \mathbb{V}_\mathcal{O}$, when 
\begin{equation*} 
	w_{ik} = w_{jk}, \ \forall k \in \mathcal{N}_O(i) = \mathcal{N}_O(j).
\end{equation*}
More precisely, the condition for checking the 0-dissimilarity between a pair $i,j$ is given as follows. 
\begin{pro} \label{pro:0dis}
	Consider the network system $\bm{\Sigma}$ with $p$ inputs, clusterable vertices $i$ and $j$ are 0-dissimilar if there exists a scalar $\beta$ such that
	\begin{equation}\label{eq:0dis}
		\mathbf{e}_{ij}^T \begin{bmatrix}
			F & \mathcal{L} - \beta I 
		\end{bmatrix} = 0, 
		\ \text{or} \ 
		\begin{bmatrix}
			H M^{-1} & \mathcal{L} - \beta I  
		\end{bmatrix}\mathbf{e}_{ij} = 0.
	\end{equation} 
\end{pro}

\begin{proof}
	By definition, vertices $i$ and $j$ are 0-dissimilar if $\mathcal{D}^I_{ij} = 0$ or $\mathcal{D}^O_{ij} = 0$. Thus, we first derive the sufficient condition for $\mathcal{D}^I_{ij} = 0$, i.e.,
	\begin{equation}  \label{eq:int:dissim}
		\int_{0}^{\infty} \mathbf{e}_{ij}^T  (e^{-\mathcal{L}t}-\mathcal{J})FF^T(e^{-\mathcal{L}^Tt} - \mathcal{J}^T) \mathbf{e}_{ij} dt = 0,
	\end{equation}
	which is equivalent to $\mathbf{e}_{ij}^T  (e^{-\mathcal{L}t}-\mathcal{J})F = \mathbf{e}_{ij}^T  e^{-\mathcal{L}t}F= 0$, $\forall t \in \mathbb{R}_+$. 
	It follows from the first equation in (\ref{eq:0dis}) that $\forall k\geq 1$,
	\begin{equation}
	\begin{split}
		\mathbf{e}_{ij}^T \mathcal{L}^k &= \mathbf{e}_{ij}^T (\mathcal{L} - \beta I + \beta I) \mathcal{L}^{k-1}
		\\
		&= \beta \mathbf{e}_{ij}^T \mathcal{L}^{k-1} = \cdots = \beta^{k-1} \mathbf{e}_{ij}^T \mathcal{L} = \beta^k \mathbf{e}_{ij}^T.
	\end{split}
	\end{equation}	
	As a result, the Taylor expansion of $\mathbf{e}_{ij}^T  e^{-\mathcal{L}t}F$  yields
	\begin{equation}
	\begin{split}
		\mathbf{e}_{ij}^T e^{-\mathcal{L}t}F & = \sum_{k=0}^{\infty} \dfrac{(-t)^k}{k!} \mathbf{e}_{ij}^T \mathcal{L}^k F 
		\\
		& = \mathbf{e}_{ij}^T F + \sum_{k=1}^{\infty} \dfrac{(-t)^k \beta^k}{k!} \mathbf{e}_{ij}^T F = 0.
	\end{split}
	\end{equation}
	Therefore, (\ref{eq:int:dissim}) holds.
	Analogously, the second equation in (\ref{eq:0dis}) implies that 
	$H(e^{-\mathcal{L}t}-\mathcal{J})M^{-1} \mathbf{e}_{ij} = H e^{-\mathcal{L}t} M^{-1} \mathbf{e}_{ij} = 0$,
	$\forall t \in \mathbb{R}_+$. Thus, $\mathcal{D}^O_{ij} = 0$.	
\end{proof}

Thereby, the following concept is proposed.
\begin{defn}
	The network system $\bm{\Sigma}$ in (\ref{sys}) is called a \textbf{minimal network realization} if all the vertices are reachable and detectable, and there does not exist 0-dissimilar vertices in the underlying network. 
\end{defn}
Parallel to the minimal realization of general linear systems obtained by Kalman decomposition, a minimal network realization is acquired by removing unreachable and undetectable vertices and aggregating 0-dissimilar vertices.  The following result shows that no approximation error is generated in the realization of the network minimality.

\begin{thm} \label{thm:FullyCluterable}
	If the reduced-order model $\bm{\hat{\Sigma}}$ is obtained by 
	\begin{itemize}
		\item removing all unreachable or undetectable vertices
		
		\item or merging all 0-dissimilarity vertices (i.e., $\mathcal{D}_{ij} = 0$),
	\end{itemize}
	then $\lVert \bm{{\Sigma}} - \bm{\hat{\Sigma}} \rVert_{\mathcal{H}_2} = 0$.
\end{thm}
\begin{proof}
	Neglecting the vertices that are not reachable or detectable is equivalent to remove uncontrollable or observable states of $\bm{\Sigma}$. Therefore, the reduction will not change the transfer function of the system, namely, $\lVert \bm{{\Sigma}} - \bm{\hat{\Sigma}} \rVert_{\mathcal{H}_2} = 0$.
	
	Next, we show that clustering all 0-dissimilarity vertices does not generate an approximation error neither. 
	To this end, we denote
	\begin{equation}
		\chi(s): = (sI_n + \mathcal{L})^{-1}, \hat{\chi}(s): = (sI_r +\hat{\mathcal{L}})^{-1}.
	\end{equation}
	and let $\tilde{\Gamma} : = I - \Pi \Pi^\dagger$, which satisfies
	\begin{equation} \label{eq:proGamma}
		\tilde{\Gamma}^T = \tilde{\Gamma}, \ \text{and} \ \tilde{\Gamma}^2 = \tilde{\Gamma}.
	\end{equation}
 
	Consider the error system $\bm{\Sigma}_e = \bm{\Sigma} - \bm{\hat{\Sigma}}$, whose transfer function is $\eta(s) - \hat{\eta}(s) : = H \eta_e F $ with 
	%
	\begin{equation}
	\eta_e
	 = \begin{bmatrix}
	I & -\Pi
	\end{bmatrix} 
	\begin{bmatrix}
	\chi(s) & 0\\
	0 & \hat{\chi}(s)
	\end{bmatrix} 
	 \begin{bmatrix}
	I\\
	\Pi^\dagger
	\end{bmatrix},
	\end{equation}
	where $\Pi^\dagger$ is the reflexive generalized inverse of $\Pi$ in (\ref{projection}).
 	Then, using the following nonsingular matrices 
	\begin{equation}
	\begin{split}
	T_1=\left[\begin{matrix}
	I & \Pi^\dagger\\
	0 & I_r  
	\end{matrix}\right], \  
	T_1^{-1}=\left[\begin{matrix}
	I_n  & -\Pi^\dagger\\
	0 & I_r
	\end{matrix}\right], 
	\\
	T_2=\left[\begin{matrix}
	I_n & 0\\
	\Pi^\dagger & I_r  
	\end{matrix}\right], \  
	T_2^{-1}=\left[\begin{matrix}
	I_n & 0\\
	-\Pi^\dagger  & I_r
	\end{matrix}\right], \\
	\end{split}
	\end{equation}	
	we further obtain
	\begin{equation*}
	\begin{split}
	\eta_e =&
	\begin{bmatrix}
	I & -\Pi
	\end{bmatrix}T_1 \left(T^{-1}
	\begin{bmatrix}
		\chi(s) & 0\\
		0 & \hat{\chi}(s)
	\end{bmatrix} T_1\right)
	 \cdot T_1^{-1}
	\begin{bmatrix}
	\chi(s)^{-1} & 0\\
	0 & \hat{\chi}(s)^{-1}
	\end{bmatrix}
	T_2
	\\
	& \cdot \left(T_2^{-1} 	
	\begin{bmatrix}
		\chi(s) & 0\\
		0 & \hat{\chi}(s)
	\end{bmatrix}
	T_2 \right)T_2^{-1}
	 \begin{bmatrix}
	I\\ \Pi^\dagger
	\end{bmatrix}  
	\\
	= & \chi(s) \begin{bmatrix}
	I &  -\tilde{\Gamma} \mathcal{L} \hat{\chi}(s)
	\end{bmatrix} 
	\cdot
	\begin{bmatrix}
	\chi(s)^{-1} - \Pi \hat{\chi}(s)\Pi^\dagger & -\Pi \hat{\chi}(s)\\
	\hat{\chi}(s) \Pi^\dagger & \hat{\chi}(s)
	\end{bmatrix}
	  \cdot \begin{bmatrix}
	I &  \hat{\chi}(s)\Pi^\dagger \mathcal{L} \tilde{\Gamma} 
	\end{bmatrix} \chi(s) \\
	= &  \tilde{\Gamma}\left[  \chi(s)^{-1} - \mathcal{L} \Pi \hat{\chi}(s) \Pi^\dagger \mathcal{L} \right] \tilde{\Gamma}.
	\end{split}
	\end{equation*}
	where the property in (\ref{eq:proGamma}) is used to obtain the above equation. Clearly, $\lVert \eta_e \rVert_{\mathcal{H}_2} = 0$ if 
	\begin{equation} \label{eq:GammaChi}
		\tilde{\Gamma} \chi(s) F = 0 \ \text{or} \ H \chi(s) \tilde{\Gamma} = 0.
	\end{equation}
	
	To prove (\ref{eq:GammaChi}), we assume, without loss of generality,
	that vertices $1,2,\cdots,k$ are 0-dissimilar, namely, $\mathbf{e}_1^T\chi(s)F = \mathbf{e}_2^T\chi(s)F = \cdots = \mathbf{e}_k^T\chi(s)F$, or $H\chi(s)M^{-1}\mathbf{e}_1 = H\chi(s)M^{-1}\mathbf{e}_2 = \cdots = H\chi(s)M^{-1}\mathbf{e}_k$, which means
	\begin{equation}
	\begin{split}
		 \chi(s)F: = \begin{bmatrix}
		\mathds{1}_k \mathbf{e}_1^T\chi(s)F  \\  *
		\end{bmatrix}  
		\ \text{or} \
		 H\chi(s): = \begin{bmatrix}
		H\chi(s)M^{-1}\mathbf{e}_1 \mathds{1}_k^T & *
		\end{bmatrix} M.
	\end{split}
	\end{equation}
	When all the 0-dissimilar vertices are assigned into a single cell, the following block matrices are obtained. 
	\begin{equation}
	\Pi = \begin{bmatrix}
	\mathds{1}_k & \\ & I_{n-k}
	\end{bmatrix},
	\
	M = \begin{bmatrix}
	M_a & \\ & M_b
	\end{bmatrix}, M_a \in \mathbb{R}^{k \times k}
	\end{equation}  
	where $\Pi$ is the characteristic matrix of the partition. When the vertices $1,2,\cdots,k$ are 0-input-dissimilar, we obtain
	\begin{equation*}
	\begin{split}
		\tilde{\Gamma} \chi(s) F & = 
		\begin{bmatrix}
		I_k - \dfrac{\mathds{1}_k \mathds{1}_k^T M_a} {\mathds{1}_k^T M_a \mathds{1}_k} & \\ & 0
		\end{bmatrix}
		\begin{bmatrix}
		\mathds{1}_k \mathbf{e}_1^T\chi(s)F  \\  *
		\end{bmatrix} 
		\\
		& = \begin{bmatrix}
		\left(\mathds{1}_k - \dfrac{\mathds{1}_k \mathds{1}_k^T M_a \mathds{1}_k} {\mathds{1}_k^T M_a \mathds{1}_k}\right) \mathbf{e}_1^T\chi(s)F & \\ & 0
		\end{bmatrix} = 0.
	\end{split}
	\end{equation*}
	When the vertices $1,2,\cdots,k$ are 0-output-dissimilar, we also have
	\begin{equation*}
	\begin{split}
		H \chi(s) \tilde{\Gamma} & = 
		\begin{bmatrix}
		H\chi(s)M^{-1}\mathbf{e}_1 \mathds{1}_k^T M_a & *
		\end{bmatrix}
		 \begin{bmatrix}
		I_k - \dfrac{\mathds{1}_k \mathds{1}_k^T M_a} {\mathds{1}_k^T M_a \mathds{1}_k} & \\ & 0
		\end{bmatrix}
		\\
		& = \begin{bmatrix}
		H\chi(s)M^{-1}\mathbf{e}_1\left(\mathds{1}_k^T M_a - \dfrac{\mathds{1}_k^T M_a \mathds{1}_k \mathds{1}_k^T M_a} {\mathds{1}_k^T M_a \mathds{1}_k}\right) & \\ & 0
		\end{bmatrix} 
		 = 0.
	\end{split}
	\end{equation*}
	Therefore, when clustering 0-dissimilar vertices, the equation (\ref{eq:GammaChi}) holds, which yields that $\lVert \eta_e \rVert_{\mathcal{H}_2} = \lVert \bm{{\Sigma}} - \bm{\hat{\Sigma}} \rVert_{\mathcal{H}_2} = 0$.
\end{proof}
Introducing the concept of the minimal network realization is to facilitate the computation of a reduced-order model. As finding the minimal network takes much less effort than solving the pseudo Gramians and calculating vertex dissimilarities, see the simulation results in Section \ref{sec:example}.
\begin{rem}
	It is worth mentioning that the minimal network realization is different from the concept of minimality in the sense of the controllability and observability of a state-space model. In a minimal network realization, all the vertices should not only be reachable and  delectable but also have a certain dissimilarity. The minimal network is defined in the viewpoint of graph topologies. It, however, does not necessarily mean that the vertex states have to be controllable and observable. 
\end{rem}
The following proposition concludes the relation between the two types of minimal realizations, namely, the classical minimality is sufficient for network minimality, but not vice versa. 
\begin{pro}
	If the network system $\bm{\Sigma}$ is minimal, then $\bm{\Sigma}$ is also a minimal network realization.
\end{pro} 
\begin{proof}
	The system $\bm{\Sigma}$ in (\ref{sys}) is minimal means that it is both controllable and observable. Consider the pseudo Gramians of $\bm{\Sigma}$, $\mathcal{P}$ and $\mathcal{Q}$, which are defined in (\ref{defn:PseudoGramians}).
	Thereby, we obtain from Theorem \ref{thm:controlrank},  (\ref{eq:VPV}) and (\ref{eq:UQU}) that  
	\begin{equation} \label{eq:nullP}
		\mathsf{rank}(\mathcal{P}) = n-m, \ 
		\text{and} \
		\mathsf{ker}(\mathcal{P}) = \mathsf{im}(V),
	\end{equation}
	and 
	\begin{equation} \label{eq:nullQ}
		\mathsf{rank}(\mathcal{Q}) = n-m, \ 
		\text{and} \
		\mathsf{ker}(\mathcal{Q}) = \mathsf{im}(U),
	\end{equation}
	where $m$ is the number of LSCCs in $\mathcal{G}$, and $U, V \in \mathbb{R}^{n \times m}$ are defined in (\ref{eq:UVdef}). 
	
	Now, assume that (\ref{sys}) is not a minimal network realization, i.e., it contains at least a pair of vertices that are clusterable and 0-dissimilar. Thus, 
	\begin{equation*}
		\mathcal{D}^I_{ij} = \sqrt{\mathbf{e}_{ij}^T\mathcal{P}\mathbf{e}_{ij}} = 0 \ \text{or} \  \mathcal{D}^O_{ij}  = \sqrt{\mathbf{e}_{ij}^T M^{-1} \mathcal{Q}M^{-1}\mathbf{e}_{ij}} = 0,
	\end{equation*} 
 	which means that $\mathbf{e}_{ij} \in \mathsf{ker}(\mathcal{P})$ or $\mathbf{e}_{ij} M^{-1} \in \mathsf{ker}(\mathcal{Q})$. 
 	However, Lemma \ref{lem:clusterability} implies that for clusterable vertices $i,j$, $\mathbf{e}_{ij}^T U = 0$ and $\mathbf{e}_{ij}^T M^{-1} V = 0$. They mean that
 	$\mathbf{e}_{ij} \notin \mathsf{im}(U)$ and $\mathbf{e}_{ij} \notin \mathsf{im}(M^{-1} V) $, which
	 cause contradictions to (\ref{eq:nullP}). 
	
	Therefore, 0-dissimilar vertices cannot exist in a directed network system that is controllable and observable.
\end{proof}
A simple counterexample for the converse statement is given by a network example with only two vertices:
\begin{equation}
\mathcal{L} = \begin{bmatrix}
	1 & -1 \\ -1 & 1
\end{bmatrix}, \ F = H^T = \begin{bmatrix}
1 & -1
\end{bmatrix}.
\end{equation}
Both vertices are reachable and detachable by Definition \ref{defn:reachable}, and the dissimilarity of the two vertices is $\mathcal{D}_{12} = 1$. Thus, this system is a minimal network realization. However, $\mathsf{rank} \left[ F, -LF \right] = \mathsf{rank} \left[H^T, - L^T H^T  \right] = 1$ implies that the system is uncontrollable and unobservable.

\subsection{Clustering Algorithm \& Error Computation}

This subsection provides the algorithm for selecting an appropriate clustering for the network system $\bm{\Sigma}$. To this end, a \textit{distance graph} of the network system $\bm{\Sigma}$ is defined. 
Let $\mathcal{X}$ be a matrix whose $(i,j)$-entry is given by
\begin{equation} \label{eq:Xij}
	\mathcal{X}_{ij} = \left\{
	\begin{array}{ll}
		\mathcal{D}_{ij}^{-1}, & \text{if vertices $i,j$ are clusterable;} \\
		0, & \text{otherwise.}
	\end{array}
	\right.	
\end{equation}
A larger value of $\mathcal{X}_{ij} \in \mathbb{R}_+$ indicates a higher similarity between vertices $i,j$. Clearly, $\mathcal{X}$ is nonnegative, symmetric and has zero diagonal entries. Thus, it can be seen as a weighted adjacency matrix that characterizes an undirected disconnected graph, namely a distance graph, denoted by $\mathcal{G}_D$. Denote the Laplacian matrix of $\mathcal{G}_D$ by
\begin{equation} \label{eq:undirectedLaplacian}
	L_D = L_D^T = \mathsf{diag}(\mathcal{X} \mathds{1}) - \mathcal{X}.
\end{equation}
Note that $\mathcal{G}_D$ shares the same vertex set $\mathbb{V}$ with the original directed network $\mathcal{G}$. From e.g., \cite{Godsil2013AlgebraicGraph}, the \textit{rank} of the undirected graph $\mathcal{G}_D$ is defined by $n - c$ with $n = |\mathbb{V}|$ and $c$ the number of connected components of $\mathcal{G}_D$ which satisfies $c = \mathsf{rank}(L_D)$. 

The idea of clustering algorithm is to remove certain edges of $\mathcal{G}_D$ such that it leaves $r$ connected components, i.e., $\mathbb{V}$ is partitioned into $r$ cells. Then we aggregate vertices in each cell as one vertex. Denote $\Omega: = \mathsf{diag}(\omega_1, \omega_1, \cdots, \omega_{n_e})$, where $n_e$ is the number of the edges of $\mathcal{G}_D$, and $\omega_1 \geq \omega_2 \geq \cdots \geq \omega_{n_e} > 0$ are the descending sort of the edge weights $\mathcal{X}_{ij}$. Let $\mathcal{B} \in \mathbb{R}^{n \times n_e}$ be the incidence matrix of $L_D$, then the graph clustering problem becomes finding a matrix partition 
\begin{equation} \label{eq:partitionLD}
	 L_D = \mathcal{B} \Omega \mathcal{B}^T = 
	 \begin{bmatrix}
	 	\mathcal{B}_1 & \mathcal{B}_2
	 \end{bmatrix}
	 \begin{bmatrix}
	 	\Omega_1 & 0 \\ 0 & \Omega_2
	 \end{bmatrix}
	 \begin{bmatrix}
	 	\mathcal{B}_1^T \\ \mathcal{B}_2^T
	 \end{bmatrix}
\end{equation}
such that $\mathsf{rank}(\mathcal{B}_1 \Omega_1 \mathcal{B}_1^T) = n - r$, where $r$ is the desired reduction order. $\mathcal{B}_2$ corresponds the edges of $\mathcal{G}_D$ with lower weights, namely, the edges connecting higher dissimilar vertices, and thus these edges are supposed to be removed. Thereby, the partition of $\mathbb{V}$ is generated by assign the vertices that are connected by the edges indicated by $\mathcal{B}_1$ into the same cell. Specifically, we describe the proposed clustering-based model reduction method as follows.

%
%
%
%
%
%

\begin{algorithm} 
	\caption{Clustering-Based Reduction Algorithm}
	\begin{algorithmic}[1] 
		\Input $\mathcal{L}$, $F$ and $H$, desired order $r$
		\Output $\hat{\mathcal L}$, $\hat{F}$, and $\hat{H}$
		
		\State Remove all the unreachable and undetectable vertices. 
		
		\State Find and merge the 0-dissimilar vertices by Proposition \ref{pro:0dis}.
		
		\State Compute the pseudo Gramians $\mathcal{P}$ and $\mathcal{Q}$ using Theorem \ref{thm:GenGram} and Corollary \ref{coro:GenGram}.
		
		\State Calculate the dissimilarity between clusterable vertices by Theorem \ref{thm:normcompute}, and construct the Laplacian matrix $L_D$ in (\ref{eq:undirectedLaplacian}) for the distance graph.
  
		\State Find a matrix partition (\ref{eq:partitionLD}) such that $$\mathsf{rank}(\mathcal{B}_1 \Omega_1 \mathcal{B}_1^T) = n - r.$$
	 
		\State Compute $\Pi \in \mathbb{R}^{n \times r}$ according to $\mathcal{B}_1$.
		
		\State Generate $\hat{\mathcal{L}}$ and $\hat{F}$ using the projection $\Pi$ as in (\ref{sysr}).
	\end{algorithmic}
	\label{alg}
\end{algorithm}

%
%
%
%
%

By Algorithm \ref{alg}, a reduced-order network system $\bm{\hat{\Sigma}}$ is obtained, which achieves a bounded approximation error, see Theorem \ref{thm:errbounded}. Based on the proposed pseudo Gramians in Section \ref{sec:Gramian}, the following theorem then provides an efficient method to compute the $\mathcal{H}_2$ error between the original and the reduced-order network systems.  
\begin{thm} \label{thm:errcomp}
	Let $\mathcal{P}_o$ and $\mathcal{P}_r$ be the pseudo controllability Gramians of the full-order and reduced-order network systems, $\bm{\Sigma}$ and $\bm{\hat{\Sigma}}$, respectively. Then, the approximation error between $\bm{\Sigma}$ and $\bm{\hat{\Sigma}}$ is computed as
	\begin{equation} \label{eq:errcomp}
		 \lVert \bm{\Sigma} - \bm{\hat{\Sigma}} \rVert_{\mathcal{H}_2}^2 =
		 \mathsf{tr} \left[H(\mathcal{P}_o  + \Pi\mathcal{P}_r\Pi^T - 2 \Pi\mathcal{P}_x){H}^T\right],
	\end{equation} 
	where $\mathcal{P}_x: = \tilde{\mathcal{P}}_x - \Pi^\dagger \mathcal{J} \Pi \tilde{\mathcal{P}}_x \mathcal{J}^T \in \mathbb{R}^{r \times n}$ with $ \tilde{\mathcal{P}}_x$ an arbitrary symmetric solution of the following Sylvester equation:
	\begin{equation} \label{eq:SylvX}
		\hat{\mathcal{L}}^T \tilde{\mathcal{P}}_x + \tilde{\mathcal{P}}_x  \mathcal{L}^T  - \Pi^\dagger (I - \mathcal{J}) F F^T (I - \mathcal{J}^T) = 0.
	\end{equation}
\end{thm}
\begin{proof}
	To derive the approximation error, we consider the following system.
	\begin{equation} \label{syserr}
		\bm{\Sigma_\mathsf{e}} : \left \{
		\begin{array}{l}
		\dot{\omega}(t) = \mathcal{A}(t) \omega(t) + \mathcal{B} u(t), \\
		\delta(t) =  \mathcal{C} \omega(t),
		\end{array}
		\right.
	\end{equation}
	where the state vectors ${\omega}(t)^T: =  \left[x(t)^T, \hat{x}(t)^T\right]^T$, $\delta(t) : = y(t) - \hat{y}(t)$, and
	\begin{equation*} 
		\mathcal{A} = -\left[ \begin{matrix}
		\mathcal{L} & 0 \\ 0 & \hat{\mathcal L}
		\end{matrix}\right],\ 
		\mathcal{F} = \left[ \begin{matrix}
		F \\ \hat{F}
		\end{matrix}\right],\ 
		\mathcal{H} = \begin{bmatrix}
		H & -\hat{H}
		\end{bmatrix}.
	\end{equation*}
	Obviously, $\lVert \bm{\Sigma_\mathsf{e}} \rVert_{\mathcal{H}_2} = \lVert \bm{\Sigma} - \bm{\hat{\Sigma}} \rVert_{\mathcal{H}_2}$. Observe that $\bm{\Sigma_\mathsf{e}}$ is a semistable system
	due to $\mathcal{A}$ matrix in (\ref{syserr}), and 
	by Theorem \ref{thm:errbounded}, the $\mathcal{H}_2$ norm of $\bm{\Sigma_\mathsf{e}}$ is bounded. From Lemma \ref{lem:trace}, we obtain
	\begin{equation} 
	\lVert \bm{\Sigma_\mathsf{e}} \rVert_{\mathcal{H}_2}^2 = \mathsf{tr}(\mathcal{H} \mathcal{P}_e \mathcal{H}^T),
	\end{equation} 
	where $\mathcal{P}_e$ is the pseudo controllability Gramian of the error system $\bm{\Sigma_\mathsf{e}}$.

	To obtain $\mathcal{P}_e$, we refer to Theorem \ref{thm:GenGram} that $\mathcal{P}_e$ is a solution of the Lyapunov equation
	\begin{equation} \label{eq:LyapEqErrSys}
		\mathcal{A} \tilde{\mathcal{P}}_e +  \tilde{\mathcal{P}}_e \mathcal{A}^T + (I-\mathcal{J}_e) \mathcal{F} \mathcal{F}^T (I-\mathcal{J}_e^T) = 0.
	\end{equation}
	Here, $\mathcal{J}_e : = \text{blkdiag}(\mathcal{J}, \hat{\mathcal{J}})$ where $\hat{\mathcal{J}} = \Pi^\dagger \mathcal{J} \Pi$ is implied by (\ref{eq:thmkey}). 
	Let
	\begin{equation}
		\tilde{\mathcal{P}}_e = \begin{bmatrix}
		\tilde{\mathcal{P}}_o & \tilde{\mathcal{P}}_x^T \\\tilde{\mathcal{P}}_x & \tilde{\mathcal{P}}_r
		\end{bmatrix}, 
	\end{equation}
	with $\tilde{\mathcal{P}}_o \in \mathbb{R}^{n \times n}$ and $\tilde{\mathcal{P}}_r \in \mathbb{R}^{r \times r}$. Accordingly, (\ref{eq:LyapEqErrSys}) is partitioned into three equations:
	\begin{subequations} \label{eq:3LyapEqs}
	\begin{empheq}[left=\empheqlbrace]{align}
	\mathcal{L} \tilde{\mathcal{P}}_o + \tilde{\mathcal{P}}_o \mathcal{L}^T - (I_n-\mathcal{J})FF^T(I_n-\mathcal{J}^T) &= 0,  \label{eq:LyapO} 
	\\
	\hat{\mathcal{L}} \tilde{\mathcal{P}}_r + \tilde{\mathcal{P}}_r \hat{\mathcal{L}}^T - (I_r-\hat{\mathcal{J}})\hat{F}\hat{F}^T(I_r-\hat{\mathcal{J}}^T) &= 0,
	\label{eq:LyapR} 
	\\
	\hat{\mathcal{L}} \tilde{\mathcal{P}}_x + \tilde{\mathcal{P}}_x {\mathcal{L}}^T - (I_r-\hat{\mathcal{J}})\hat{F}{F}^T(I_n-{\mathcal{J}}^T) &= 0,
	\label{eq:LyapX} 
	\end{empheq}
	\end{subequations}
	where (\ref{eq:LyapX}) is equivalent to (\ref{eq:SylvX}) due to $\hat{F} = \Pi^\dagger F$ and $\hat{\mathcal{J}} \Pi^\dagger = \Pi^\dagger \mathcal{J} \Pi \Pi^\dagger = \Pi^\dagger \mathcal{J}$ (see (\ref{eq:UVprop})). Then, using Corollary \ref{coro:GenGram},
	the pseudo controllability Gramian of $\bm{\Sigma_\mathsf{e}}$ is computed.
	\begin{equation*}
	\begin{split}
		\mathcal{P}_e &= \tilde{\mathcal{P}}_e - \mathcal{J}_e \tilde{\mathcal{P}}_e \mathcal{J}_e^T 
		\\&= 
		\begin{bmatrix}
		\tilde{\mathcal{P}}_o - \mathcal{J} \tilde{\mathcal{P}}_o \mathcal{J}^T
		& 
		\tilde{\mathcal{P}}_x^T - {\mathcal{J}} \tilde{\mathcal{P}}_x^T \hat{\mathcal{J}}^T
		\\
		\tilde{\mathcal{P}}_x - \hat{\mathcal{J}} \tilde{\mathcal{P}}_x {\mathcal{J}}^T
		&
		\tilde{\mathcal{P}}_r - \hat{\mathcal{J}} \tilde{\mathcal{P}}_r \hat{\mathcal{J}}^T
		\end{bmatrix} 
		:= \begin{bmatrix}
		{\mathcal{P}}_o & {\mathcal{P}}_x^T \\{\mathcal{P}}_x & {\mathcal{P}}_r
		\end{bmatrix},
	\end{split}
	\end{equation*}
 	where $\mathcal{P}_o$ and $\mathcal{P}_r$ are the pseudo controllability Gramians of the systems $\bm{\Sigma}$ and $\hat{\bm{\Sigma}}$, respectively.
 	
 	Thereby, we evaluate the approximation error as follows. 
 	\begin{equation}
 	\begin{split}
  		\lVert \bm{\Sigma} - \bm{\hat{\Sigma}} \rVert_{\mathcal{H}_2}^2 & =
 		\mathsf{tr}\left(\begin{bmatrix}
 			H & -\hat{H}
 		\end{bmatrix}
 		\begin{bmatrix}
 		{\mathcal{P}}_o & {\mathcal{P}}_x^T \\{\mathcal{P}}_x & {\mathcal{P}}_r
 		\end{bmatrix}
 		\begin{bmatrix}
 			H^T \\ -\hat{H}^T
 		\end{bmatrix}\right)
 		\\
 		&=\mathsf{tr} (H {\mathcal{P}}_o H^T + \hat{H} {\mathcal{P}}_r \hat{H}^T - 2\hat{H}{\mathcal{P}}_x H^T),
 	\end{split}
 	\end{equation}
 	which leads to (\ref{eq:errcomp}).
\end{proof}
Notice that the error in (\ref{eq:errcomp}) can be also characterized by pseudo observability Gramians. Suppose $\mathcal{Q}_o$ and $\mathcal{Q}_r$ are the pseudo observability Gramians of $\bm{\Sigma}$ and $\bm{\hat{\Sigma}}$, respectively. Then, an alternative computation for (\ref{eq:errcomp}) is
\begin{equation}  
\lVert \bm{\Sigma} - \bm{\hat{\Sigma}} \rVert_{\mathcal{H}_2}^2 =
	\mathsf{tr} \left[F^T(\mathcal{Q}_o  + (\Pi^\dagger)^T \mathcal{Q}_r\Pi^\dagger - 2 \mathcal{Q}_x\Pi^\dagger){F}\right],
\end{equation} 
where $\mathcal{Q}_x: = \tilde{\mathcal{Q}}_x - \mathcal{J}^T \tilde{\mathcal{Q}}_x  \Pi^\dagger  \mathcal{J} \Pi \in \mathbb{R}^{n \times r}$ with $ \tilde{\mathcal{Q}}_x$ an symmetric solution of the Sylvester equation:
\begin{equation}  
	\mathcal{L}^T \tilde{\mathcal{Q}}_x +  \tilde{\mathcal{Q}}_x \hat{\mathcal{L}} - (I - \mathcal{J}^T) H^T H (I - \mathcal{J})\Pi = 0.
\end{equation}
The proof for the above statement is similar to Theorem \ref{thm:errcomp} and thus is omitted.

\section{Numerical Examples} \label{sec:example}

\subsection{Sensor Network}
A sensor network in \cite{XiaodongCDC2017Digraph} is considered and adapted to illustrate the proposed method in this paper. The topology of the studied directed network is depicted in Fig. \ref{fig:ExOrigin}, which is weakly connected and contains three LSCCs: $\{1,2,3,4,5\}$, $\{15,16\}$ and $\{20\}$. All the directional edges have identical weights that equal to $1$. The sets $\mathbb{V}_\mathcal{I}: = \{2, 7, 16\}$ and $\mathbb{V}_\mathcal{O}: = \{5, 9, 11\}$ are collections of controlled and measured vertices, respectively. The minimal network is realized by two steps. First, the unreachable vertex $20$ and  undetectable vertices $17$, $18$ and $19$. Then, three pairs of 0-dissimilar vertices, namely $\{1,5\}$, $\{8,9\}$, and $\{13,14\}$ are found by Proposition \ref{pro:0dis} and clustered to obtain Fig. \ref{fig:ExMin}, which is still weakly connected as there are two LSCCs.  

By considering the effects of both inputs and outputs, we compute the vertex dissimilarities of the minimal network using Theorem \ref{thm:normcompute}.
Then, to yield a reduced-order network system of 7 vertices, we apply Algorithm \ref{alg}, 
that leads to the following graph clustering
\begin{equation*}
	\begin{split}
		\mathcal{C}_1 &= \{1,2,3, 4\}, \
		\mathcal{C}_2 = \{5, 6\}, \ \mathcal{C}_3 = \{7\}, \\
		\mathcal{C}_4 &= \{8\}, \ \mathcal{C}_5=\{9\}, \
		\mathcal{C}_6=\{10,11\}, \
		\mathcal{C}_7=\{12,13\}.
	\end{split} 
\end{equation*}
Thus, a simplified directed network is constructed as in Fig. \ref{fig:ExRedNet1}, which indicates that the network structure is preserved in the new model. Also, the approximation error is computed: $\lVert \bm{\Sigma} - \bm{\hat{\Sigma}} \rVert_{\mathcal{H}_2} = 0.1969$. 
Next, we only use the input information for selecting clusters, namely, only the input dissimilarities are considered as the criteria for the partition of the vertices. As a result, we obtain a different clustering of the minimal network, where $\mathcal{C}_2 = \{5, 7\}$, and  $\mathcal{C}_3 = \{6\}$. This produces a simplified directed network with a different topology as shown in Fig. \ref{fig:ExRedNet2}.
In this case, the approximation error is evaluated as
$\lVert \bm{\Sigma} - \bm{\hat{\Sigma}} \rVert_{\mathcal{H}_2} = 0.2514$, which is almost $30\%$ larger than the one obtained in the former case. Thus, to better approximate the input-output behavior of a network system, the output distributions should also be considered in order to construct a more accurate reduced-order network model. 

\begin{figure}[!tp] 
	\begin{minipage}[t]{0.5\linewidth}
		\centering
		\includegraphics[scale=.6]{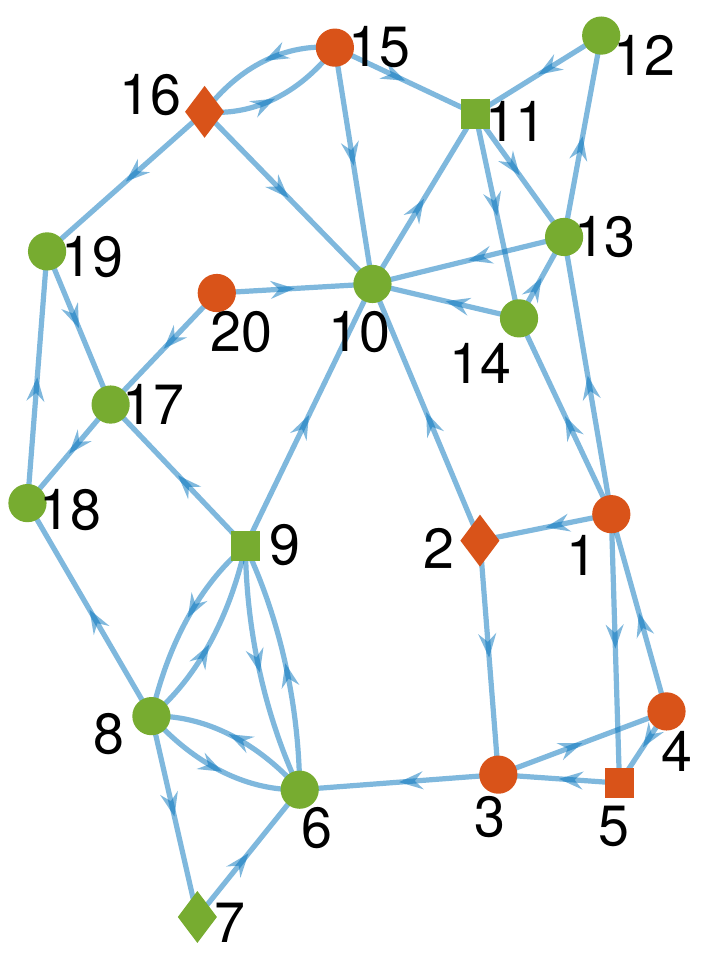}	
		\subcaption{}
		\label{fig:ExOrigin}
	\end{minipage}%
	\begin{minipage}[t]{0.5\linewidth}
		\centering
		\includegraphics[scale=.6]{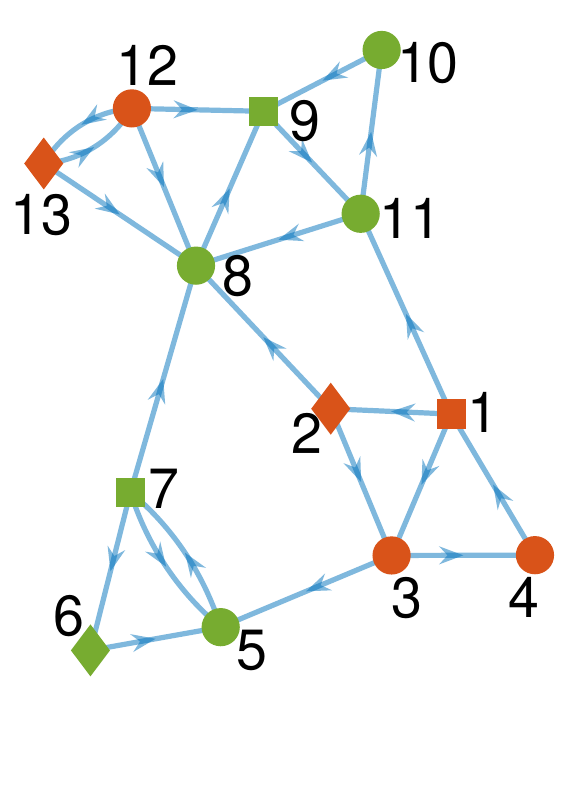}
		\subcaption{}
		\label{fig:ExMin}
	\end{minipage}
	\caption{(a) A directed sensor network consisting of 20 vertices. (b) A minimal network realization of this directed network, which is obtained by removing vertices $17$, $18$, $19$, $20$ and merging pairs $\{1,5\}$, $\{8,9\}$, $\{13,14\}$ in Fig. \ref{fig:ExOrigin}. In both figures, the controlled and measured vertices are labeled as diamonds and squares, respectively. }
\end{figure}

\begin{figure}[!tp]\centering
	\begin{minipage}[t]{0.5\linewidth}
		\centering
		\includegraphics[width=0.48\textwidth]{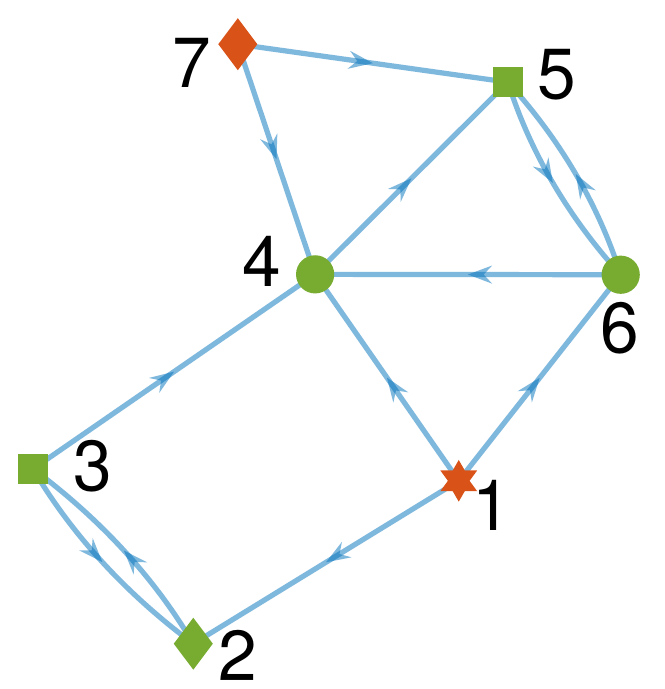}
		\subcaption{}
		\label{fig:ExRedNet1}
	\end{minipage}%
	\begin{minipage}[t]{0.5\linewidth}
		\centering
		\includegraphics[width=0.52\textwidth]{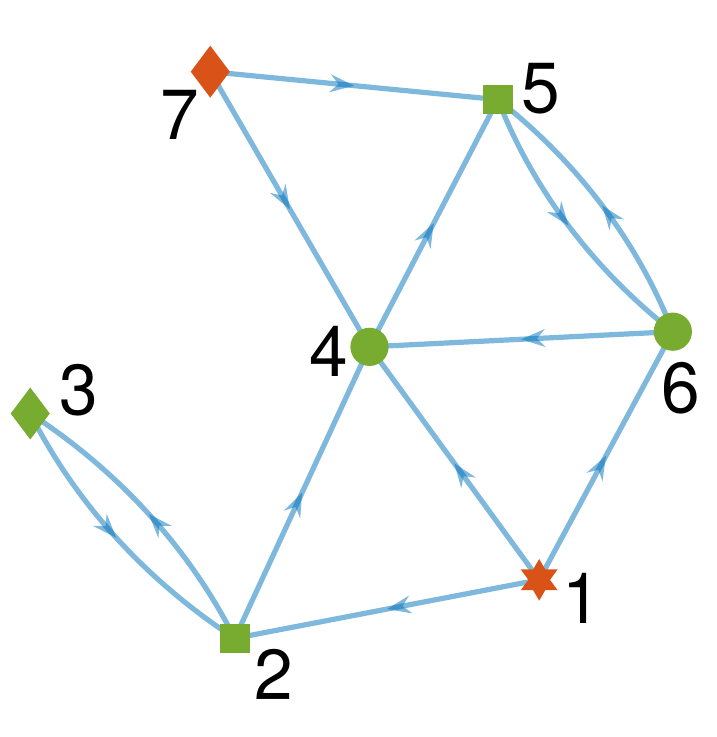}
		\subcaption{}
		\label{fig:ExRedNet2}
	\end{minipage}
	\caption{(a) The reduced directed network generated by the proposed dissimilarity-based clustering. (b) The alternative directed network obtained by only considering the input dissimilarity. In the two figures, vertex $1$ is labeled as a hexagram indicating that it is controlled and measured simultaneously.}	
\end{figure}

\subsection{Large-scale Directed Network}

The efficiency of the proposed approach is verified by
a large-scale directed network example, see Fig. \ref{fig:BigOrigin}. Note that the methods in \cite{XiaodongCDC2017Digraph,ishizaki2015clustereddirected} are not applicable for this networks, since it is not strongly connected. The data of directed weighted graph is acquired from
The Harwell-Boeing Sparse Matrix Collection, which is available at   \url{https://math.nist.gov/MatrixMarket/data/Harwell-Boeing}. In this paper, the simulation is
implemented with Matlab 2018a in the environment of 64-bit
operating system, which is equipped with Intel Core i5-3470
CPU @ 3.20GHz, RAM 8.00 GB. 

We select three controlled and three measured vertices: $\mathbb{V}_\mathcal{I}: = \{1, 152, 728\}$ and $\mathbb{V}_\mathcal{O}: = \{246,615,733\}$. For comparison purposes, we reduce the original networks by the dissimilarity-based (proposed method in this paper), input dissimilarity-based (the methods in e.g., \cite{XiaodongCDC2017Digraph}) and random clustering methods, respectively. The $\mathcal{H}_2$-norm approximation errors of the reduced models with different dimensions from 10 vertices to 700 are computed by Theorem \ref{thm:errcomp}, see the Fig. \ref{fig:BigErr}, which shows that the proposed method considering the output efforts as well generally has a better performance than the one only take into account the influence of inputs. The approximation errors of both methods decay rapidly when the reduced order $r < 100$, and they both have distinct advantage over the random clustering method, where the cells are selected randomly from the largest clusterable sets. 
Note that when $r = 100$, the approximation errors $\lVert \bm{\Sigma} - \bm{\hat{\Sigma}} \rVert_{\mathcal{H}_2} = 0.0011$, while the maximal values of input and output dissimilarities are $0.8839$ and $3.6949$, respectively. Hence, the reduced-order network with $100$ vertices provides a rather accurate approximation of the original 735-vertex network. To illustrate the efficiency of the proposed method, when producing the 100-dimensional reduced network, we record and list the computational cost for each step in Table \ref{Table}. It indicates that the majority of the computation time is spent on solving the pseudo Gramians. In contrast, the time for computing the minimal network, vertex dissimilarities and for the clustering algorithm is much less.
To be more efficient for even larger networks, e.g., $n>5000$, we can consider the Alternating Direction Implicit (ADI) method or Krylov subspace method to generate low rank approximations to the solutions of the Lyapunov equations in (\ref{eq:thm2eqs}) and (\ref{eq:thm2eqsQ}). As these numerical algorithms are more effective than the standard approach due to the sparsity of the Laplacian matrix $\mathcal{L}$. We refer to e.g., \cite{Li2002LowRank,Jaimoukha1994krylovLyap} for more details. However, we do not apply these here since we can still compute the Gramians using standard approach within a reasonable time. Hence, it is no need for such an approximation.

In Fig. \ref{fig:reducednetwork}, the topologies of  reduced-order networks with different dimensions are plotted to demonstrate the preservation of directed network structures. In conclusion, this simulation example shows that
the proposed clustering method is feasible and effective in
model reduction of large-scale directed network systems.

\begin{figure}[!tp] 
	\centering
	\includegraphics[scale=.8]{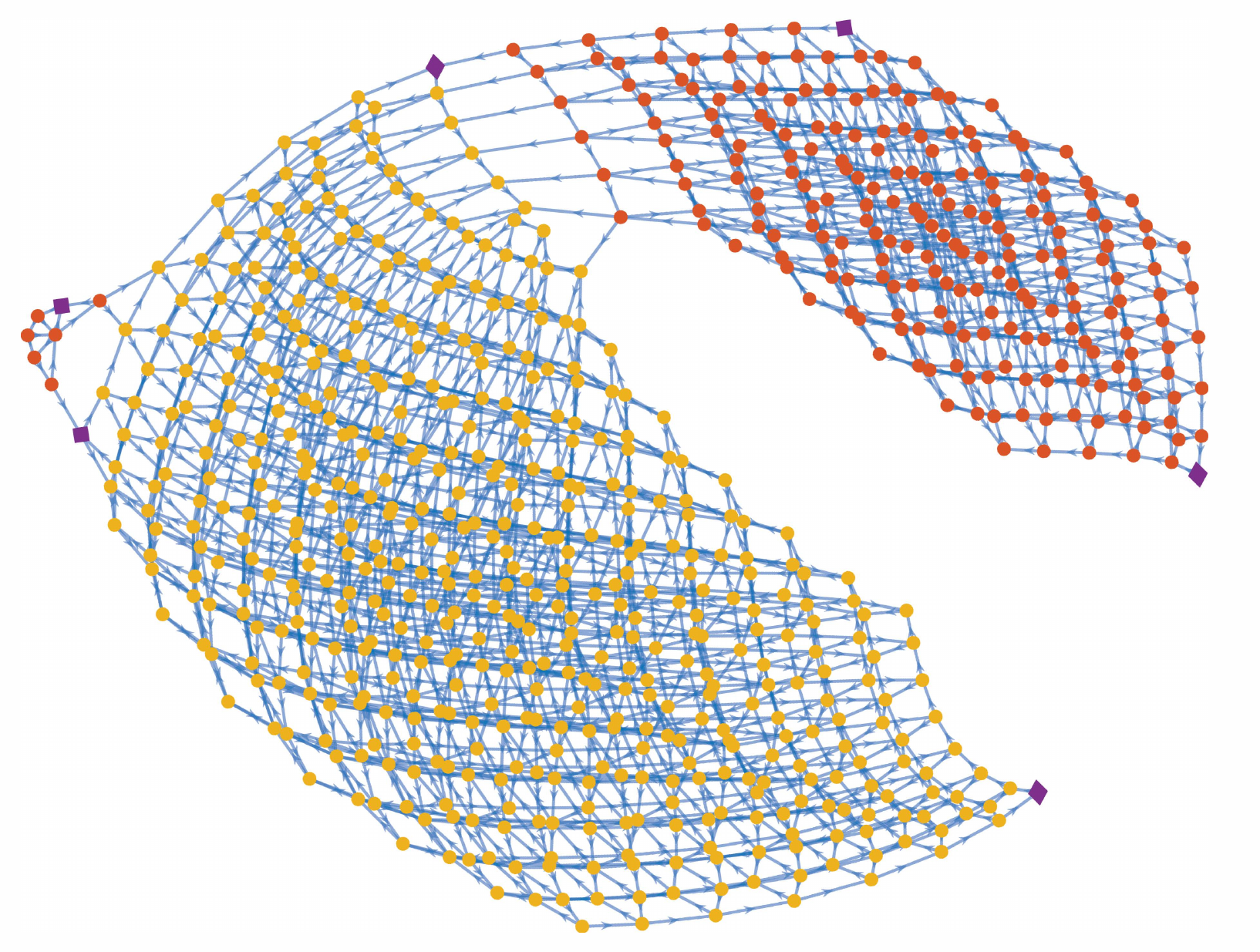}	
	\caption{A weakly connected directed network consisting of 735 vertices, where the controlled and measured vertices are labeled as diamonds and squares, respectively.}
	\label{fig:BigOrigin}
\end{figure}

\begin{figure}[!tp] 
	\centering
	\includegraphics[scale=.6]{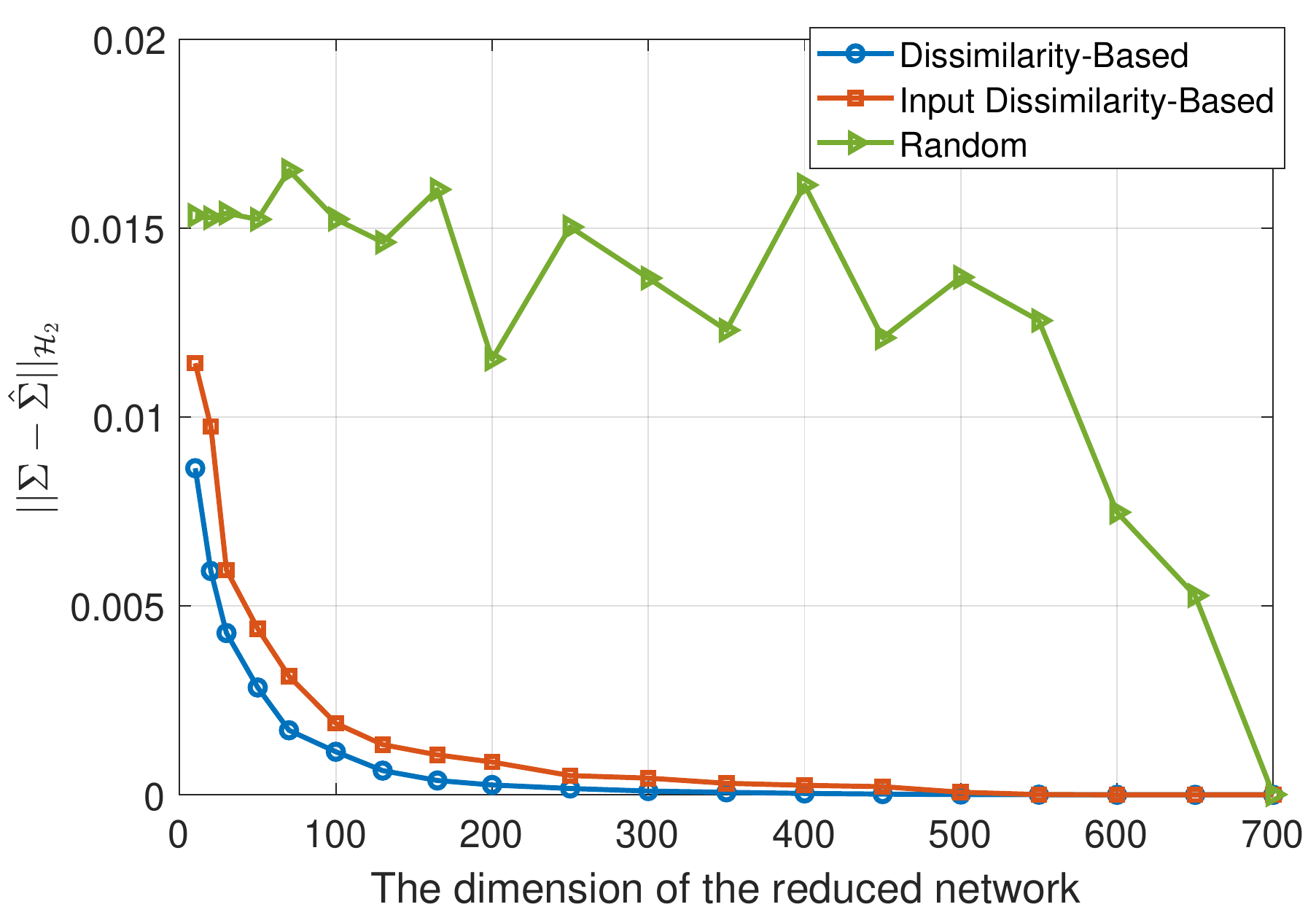}	
	\caption{Approximation error comparisons among the proposed dissimilarity-based clustering, input dissimilarity-based clustering and random clustering algorithms.}
	\label{fig:BigErr}
\end{figure}

\begin{table}[!tp]  
	\caption{Computation time for each step}
	\centering 
	\renewcommand\arraystretch{1.5}  
	\begin{tabular}{l | l} 
		\hline
		Minimal Network & 23.8949s\\
		Pseudo Gramians & 538.0895s\\ Dissimilarity & 99.5960s\\
		Graph Clustering & 0.0035s \\
		\hline  
	\end{tabular}
	\label{Table}
\end{table}

\begin{figure}[!tp]\centering
	\begin{minipage}[t]{0.5\linewidth}
		\centering
		\includegraphics[width=0.8\textwidth]{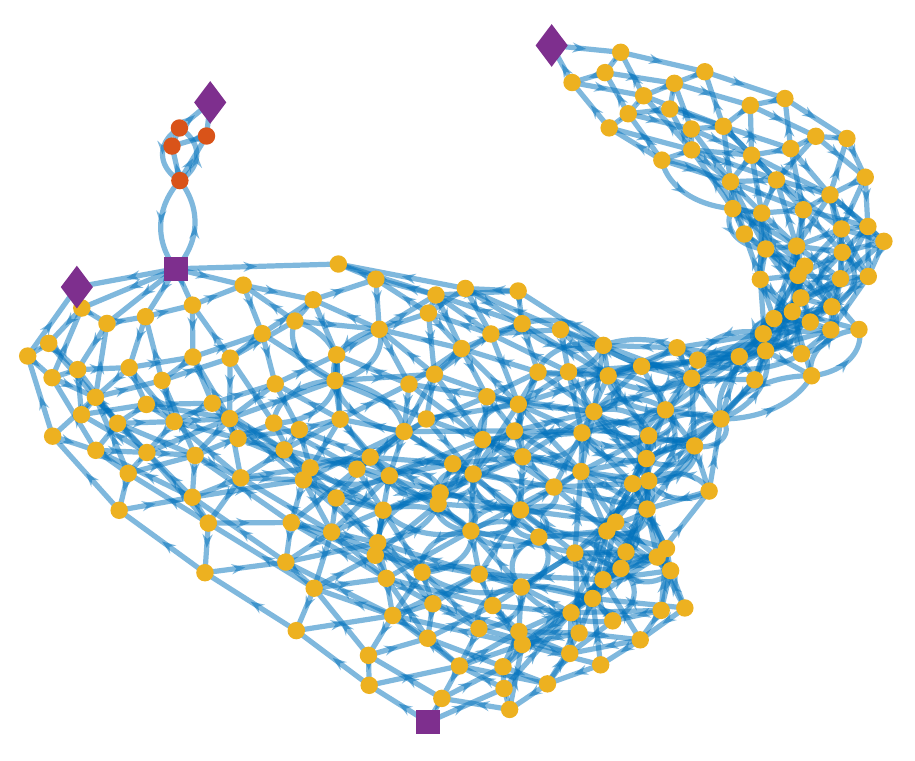}	
		\subcaption{}
	\end{minipage}%
	\hfill
	\begin{minipage}[t]{0.5\linewidth}
		\centering
		\includegraphics[width=0.8\textwidth]{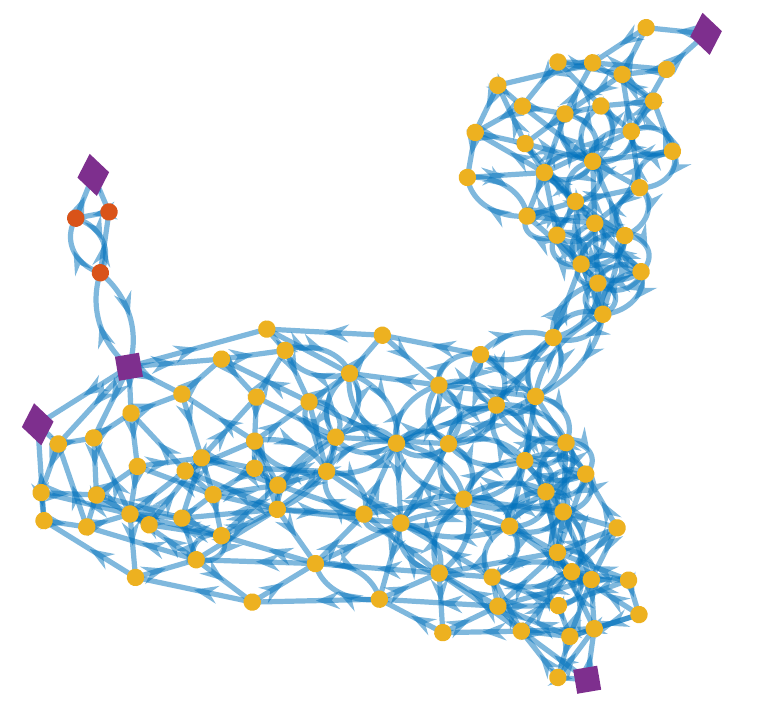}
		\subcaption{}
	\end{minipage}
	\vfill
	\begin{minipage}[t]{0.5\linewidth}
		\centering
		\includegraphics[width=0.8\textwidth]{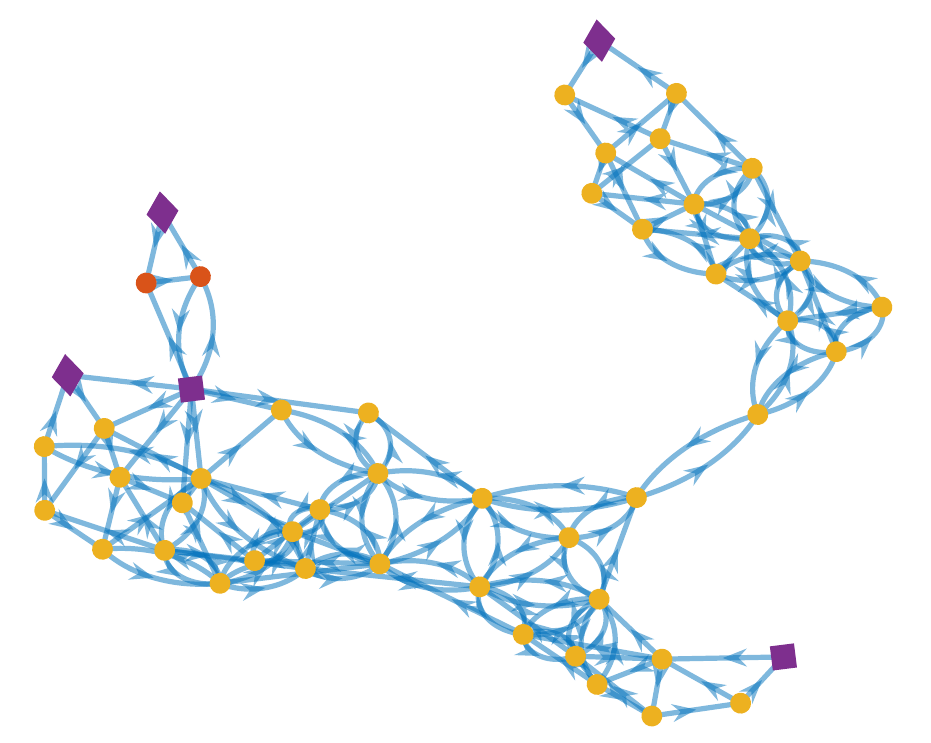}	
		\subcaption{}
	\end{minipage}%
	\hfill
	\begin{minipage}[t]{0.5\linewidth}
		\centering
		\includegraphics[width=0.7\textwidth]{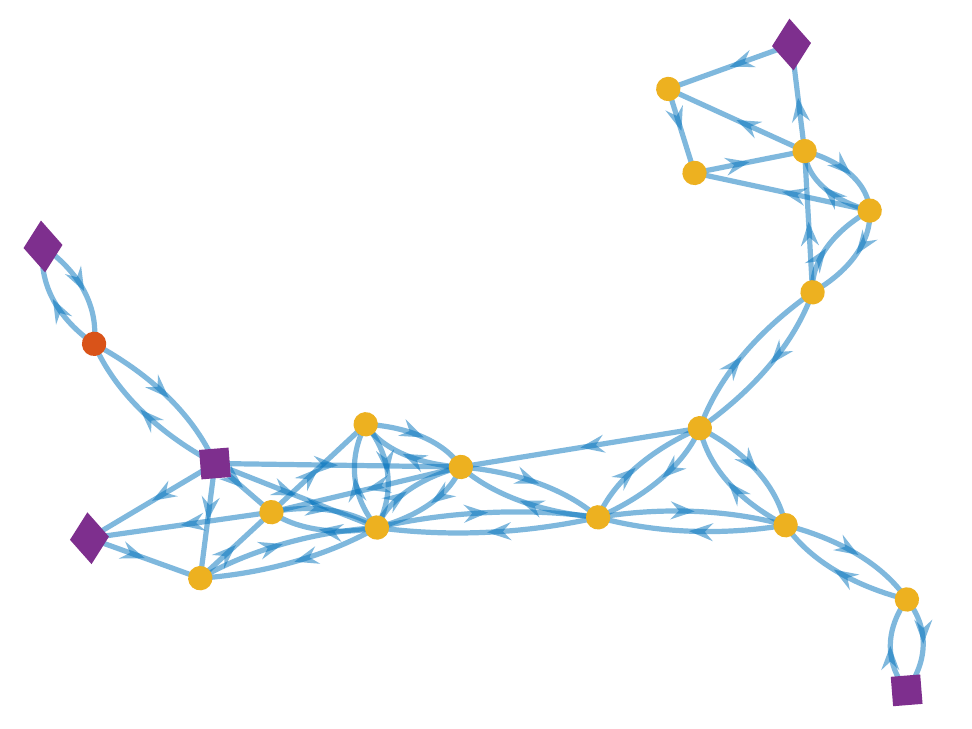}
		\subcaption{}
	\end{minipage}
	\caption{Reduced directed networks with different dimensions. (a) $r = 200$; (b) $r = 100$; (c) $r = 50$; (d) $r = 20$.}
	\label{fig:reducednetwork}		
\end{figure}
\section{Conclusion} \label{sec:conclusion}

This paper solves a structure preserving model reduction problem for directed network systems that obey locally consensus protocols and have semistable dynamics. The notion of clusterability is proposed to classify the groups of vertices that can be aggregated to guarantee a bounded approximation error. The pairwise dissimilarity, quantifying the difference between two clusterable vertices, can be characterized for a directed network based on the pseudo controllability and observability Gramians of semistable systems. A graph clustering algorithm then disassembles the vertices that behave differently. The reduced-order model is obtained in the Petrov-Galerkin framework with projections generated from the resulting clustering of the network. It is shown that the reduced-order model preserves a network structure among the clusters, as a reduced Laplacian matrix. Through numerical examples, the efficiency of the proposed method is then verified.

%
%


%

\bibliographystyle{IEEEtran}
\bibliography{NetworkReduction}

\end{document}